\documentclass[a4paper]{amsart}
\pdfoutput=1
\usepackage[utf8]{inputenc}
\usepackage[T1]{fontenc}
\usepackage{lmodern}
\usepackage{amssymb,amsxtra}
\usepackage{graphicx}
\usepackage{mathabx}
\usepackage{fancybox}
\usepackage{nicefrac,mathtools}
\usepackage[lite]{amsrefs}
\usepackage[toc,page]{appendix}
\renewcommand*{\MR}[1]{ \href{http://www.ams.org/mathscinet-getitem?mr=#1}{MR \textbf{#1}}}
\newcommand*{\arxiv}[1]{\href{http://www.arxiv.org/abs/#1}{arXiv: #1}}

\renewcommand{\PrintDOI}[1]{\href{http://dx.doi.org/\detokenize{#1}}{doi: \detokenize{#1}}}

\BibSpec{book}{%
  +{}  {\PrintPrimary}                {transition}
  +{,} { \textit}                     {title}
  +{.} { }                            {part}
  +{:} { \textit}                     {subtitle}
  +{,} { \PrintEdition}               {edition}
  +{}  { \PrintEditorsB}              {editor}
  +{,} { \PrintTranslatorsC}          {translator}
  +{,} { \PrintContributions}         {contribution}
  +{,} { }                            {series}
  +{,} { \voltext}                    {volume}
  +{,} { }                            {publisher}
  +{,} { }                            {organization}
  +{,} { }                            {address}
  +{,} { \PrintDateB}                 {date}
  +{,} { }                            {status}
  +{}  { \parenthesize}               {language}
  +{}  { \PrintTranslation}           {translation}
  +{;} { \PrintReprint}               {reprint}
  +{.} { }                            {note}
  +{.} {}                             {transition}
  +{} { \PrintDOI}                   {doi}
  +{} { available at \url}            {eprint}
  +{}  {\SentenceSpace \PrintReviews} {review}
}

\usepackage{enumitem} 
\setlist[enumerate,1]{label=\textup{(\arabic*)}}

\usepackage{tikz}
\usetikzlibrary{matrix,calc,arrows,decorations.pathmorphing}
\tikzset{node distance=2cm, auto}

\tikzset{cd/.style=matrix of math nodes,row sep=2em,column sep=2em, text height=1.5ex, text depth=0.5ex}
\tikzset{cdar/.style=->,auto}
\tikzset{mid/.style={anchor=mid}} 
\tikzset{narrowfill/.style={inner sep=1pt, fill=white}}
\tikzset{rndblock/.style={rounded corners,rectangle,draw,outer sep=0pt}}

\usepackage{tikz-cd}

\usepackage{microtype}
\usepackage[pdftitle={Noncommutative Cartan C*-subalgebras},   pdfauthor={Bartosz Kwa\'sniewski, Ralf Meyer},  pdfsubject={Mathematics}]{hyperref}

\theoremstyle{plain}
\newtheorem{theorem}{Theorem}
\newtheorem{lemma}[theorem]{Lemma}

\newtheorem{proposition}[theorem]{Proposition}
\newtheorem{corollary}[theorem]{Corollary}
\theoremstyle{definition}
\newtheorem{definition}[theorem]{Definition}
\theoremstyle{remark}
\newtheorem{remark}[theorem]{Remark}
\newtheorem{example}[theorem]{Example}

\numberwithin{theorem}{section}
\numberwithin{equation}{section}

\DeclareMathOperator{\Aut}{Aut}
\DeclareMathOperator{\Prim}{Prim}
\DeclareMathOperator{\Bis}{Bis}

\newcommand{\D}{X}
\newcommand{\haction}{h}

\newcommand*{\nb}{\nobreakdash}
\newcommand*{\Star}{\(^*\)\nobreakdash-}

\newcommand*{\C}{\mathbb C}
\newcommand*{\Z}{\mathbb Z}

\newcommand*{\N}{\mathbb N}

\newcommand*{\Ideals}{\mathbb I}
\newcommand*{\Open}{\mathbb O}
\newcommand*{\Null}{\mathcal N}
\newcommand*{\Bound}{\mathbb B}
\newcommand*{\red}{\mathrm r}
\newcommand*{\alg}{\mathrm{alg}} 

\newcommand*{\Cst}{\textup C^*}
\newcommand*{\Mult}{\mathcal M}

\newcommand*{\Cont}{\textup C}
\newcommand*{\Contb}{\textup C_\textup b}

\newcommand*{\prid}[1][p]{\mathfrak{#1}} 

\newcommand*{\Slice}{\mathcal S}
\newcommand*{\Id}{\textup{Id}}

\newcommand*{\Hils}{\mathcal H}
\newcommand*{\Hilm}[1][E]{\mathcal #1}

\newcommand*{\B}{\mathcal B}
\newcommand*{\A}{\mathcal A}
\newcommand*{\defeq}{\mathrel{\vcentcolon=}}
\newcommand*{\congto}{\xrightarrow\sim}

\DeclarePairedDelimiter{\abs}{\lvert}{\rvert}
\DeclarePairedDelimiter{\norm}{\lVert}{\rVert}
\DeclarePairedDelimiterX{\braket}[2]{\langle}{\rangle}{#1\,\delimsize\vert\,\mathopen{}#2}
\DeclarePairedDelimiterX{\BRAKET}[2]{\langle}{\rangle}{\!\delimsize\langle#1\,\delimsize\vert\,\mathopen{}#2\delimsize\rangle\!}
\DeclarePairedDelimiterX{\setgiven}[2]{\{}{\}}{#1\,{:}\,\mathopen{}#2}

\newcommand*{\dual}[1]{\widehat{#1}}

\newcommand*{\s}{s}
\newcommand*{\rg}{r}

\DeclareMathOperator*{\stlim}{s-lim}

\newcommand*{\onto}{\twoheadrightarrow}

\begin{document}
\title{Noncommutative Cartan \(\Cst\)-subalgebras}

\author{Bartosz Kosma Kwa\'sniewski}
\email{bartoszk@math.uwb.edu.pl}
 \address{Department of Mathematics\\
   University  of Bia\l ystok\\
   ul.\@ K.~Cio\l kowskiego 1M\\
   15-245 Bia\l ystok\\
   Poland}

\author{Ralf Meyer}
\email{rmeyer2@uni-goettingen.de}
\address{Mathematisches Institut\\
 Georg-August-Universit\"at G\"ottingen\\
 Bunsenstra\ss e 3--5\\
 37073 G\"ottingen\\
 Germany}

\begin{abstract}
  We characterise Exel's noncommutative Cartan subalgebras in
  several ways using uniqueness of conditional expectations,
  relative commutants, or purely outer inverse semigroup actions.
  We describe in which sense the crossed product decomposition for a
  noncommutative Cartan subalgebra is unique.  We relate the
  property of being a noncommutative Cartan subalgebra to aperiodic
  inclusions and effectivity of dual groupoids.  In particular, we
  extend Renault's characterisation of commutative Cartan
  subalgebras.
\end{abstract}

\maketitle
\setcounter{tocdepth}{1}

\section{Introduction}
\label{sec:introduction}

Many important \(\Cst\)\nb-algebras may be described as groupoid
\(\Cst\)\nb-algebras of Hausdorff, étale, locally compact groupoids.
Renault~\cite{Renault:Cartan.Subalgebras} defined a \emph{Cartan
  subalgebra} in a \(\Cst\)\nb-algebra~\(B\) as a maximal abelian,
regular \(\Cst\)\nb-subalgebra \(A\subseteq B\) with a faithful
conditional expectation \(E\colon B\to A\).  Assuming~\(B\) to be
separable, he proved that \(B\cong \Cst_\red(H,\Sigma)\) for a
topologically principal, second countable, Hausdorff, étale twisted
groupoid \((H,\Sigma)\).  In addition, the isomorphism
\(B\to \Cst_\red(H,\Sigma)\) maps~\(A\) onto \(\Cont_0(H^0)\), and
\((H,\Sigma)\) is unique up to isomorphism.  As a result, a Cartan
subalgebra allows to reconstruct an underlying dynamical system from
a \(\Cst\)\nb-inclusion.  Kumjian's earlier theory of
\(\Cst\)\nb-diagonals in~\cite{Kumjian:Diagonals} covered only the
reduced twisted groupoid \(\Cst\)\nb-algebras of \emph{principal}
étale groupoids.  The notion of a Cartan subalgebra has become
ubiquitous in the study of \(\Cst\)\nb-algebras (see, for instance,
\cite{Li-Renault:Cartan_subalgebras} and the sources cited there);
in particular, it is crucial for the UCT problem
\cites{Barlak-Li:Cartan_UCT, Barlak-Li:Cartan_UCT_II} and rigidity
of dynamical
systems~\cite{Carlsen-Ruiz-Sims-Tomforde:Reconstruction}.

The success of Renault's theory led Exel~\cite{Exel:noncomm.cartan}
to generalise Renault's definition to the case when the subalgebra
\(A\subseteq B\) need no longer be commutative.  In Exel's
definition of a \emph{noncommutative Cartan subalgebra}, Renault's
requirement that~\(A\) be maximal Abelian in~\(B\) is replaced by
the requirement that all ``virtual commutants'' of~\(A\) in~\(B\) be
trivial.  The main result of~\cite{Exel:noncomm.cartan} says that
every noncommutative Cartan \(\Cst\)\nb-inclusion \(A\subseteq B\)
into a separable \(\Cst\)\nb-algebra~\(B\) is isomorphic to an
inclusion \(A\subseteq A\rtimes_\red S\) into a reduced crossed
product \(A\rtimes_\red S\) for some action of a unital inverse
semigroup~\(S\) on~\(A\) by Hilbert \(A\)\nb-bimodules.  Such
actions of inverse semigroups are equivalent to saturated
Fell bundles over inverse semigroups (see
\cites{Buss-Meyer:Actions_groupoids, Kwasniewski-Meyer:Essential}).
And Fell bundless over inverse semigroups are slightly more general
than Fell bundles over étale groupoids (see
\cites{BussExel:Fell.Bundle.and.Twisted.Groupoids,Buss-Meyer:Actions_groupoids,
  Kwasniewski-Meyer:Essential}).  Exel~\cite{Exel:noncomm.cartan}
did not identify the inverse semigroup actions for which the
inclusion \(A\subseteq A\rtimes_\red S\) is Cartan, and he did not
study the uniqueness of such crossed product decompositions.  Here
we answer these questions and push Exel's theory much further.

Our first main result (Theorem~\ref{the:nc_Cartan}) characterises
noncommutative Cartan inclusions in a number of different ways.  We
show that an inclusion \(A\subseteq A\rtimes_\red S\) is Cartan if
and only if the action of the inverse semigroup~\(S\) on~\(A\) is
``closed'' and ``purely outer''.  The notion of pure outerness is
generalised from automorphisms to Hilbert bimodules
in~\cite{Kwasniewski-Meyer:Aperiodicity}.  In Renault's setting,
pure outerness corresponds to the effectivity of the underlying
groupoid.  The condition that the action be closed corresponds to
Hausdorffness of the groupoid.  We show that an inclusion
\(A\subseteq B\) has only trivial virtual commutants if and only if,
for each ideal~\(I\) in~\(A\), the relative commutant
\(A' \cap \Mult(I B I)\) is contained in~\(\Mult(I)\).  And a
regular \(\Cst\)\nb-inclusion \(A\subseteq B\) with a faithful
conditional expectation \(E\colon B\to A\) is Cartan if and only if,
for each ideal~\(I\) in~\(A\), \(E|_{I B I}\) is the only
conditional expectation \(I B I \to I\).  When the primitive ideal
space of \(A\) is Hausdorff, it is enough to assume that the
expectation~\(E\) itself is unique, without passing to ideals (see
Proposition~\ref{pro:unique_conditionals_commutative}).  Exel
already shows that the conditional expectation \(B\to A\) for a
noncommutative Cartan subalgebra is unique.  This is a major step in
his proof that~\(B\) is isomorphic to a reduced inverse semigroup
crossed product.  Our proof of the converse direction uses
ideas of Zarikian~\cite{Zarikian:Unique_expectations}.

In passing, we extend Exel's theory, and thus also Renault's
characterisation, to the non-separable case.  This extension amounts
to replacing frames in Hilbert bimodules by certain special
approximate units, which were already used
in~\cite{Brown-Mingo-Shen:Quasi_multipliers}.
The separability assumption in the previous works goes back to the
prototype result by Feldman and
Moore~\cite{Feldman-Moore:Ergodic_II} in the von Neumann algebraic
setting.  These results have recently been extended to the
non-separable case in~\cite{Donsig-Fuller-Pitts:Extensions} using
inverse semigroup methods.  In Renault's approach separability plays
only a role in the construction of the dual groupoid.  Namely,
Renault uses groupoids of germs, while we use transformation
groupoids.  The two constructions coincide if and only if the
transformation groupoid is effective.  And a second countable,
Hausdorff, étale groupoid is effective if and only if it is
topologically principal.

If~\(A\) is commutative, then an inverse semigroup action on~\(A\)
gives rise to a groupoid, the dual groupoid \(\dual{A} \rtimes S\).
And \(A\rtimes_{(\red)} S\) is isomorphic to the (reduced) section
\(\Cst\)\nb-algebra of a Fell line bundle over~\(\dual{A}\).
Similar results are still available if the primitive ideal
space~\(\widecheck{A}\) of~\(A\) is Hausdorff.  Two extreme cases
where this happens are when~\(A\) is simple or commutative.
If~\(\widecheck{A}\) is Hausdorff, then a noncommutative Cartan
inclusion \(A\subseteq B\) is isomorphic to the inclusion of~\(A\)
into the reduced section \(\Cst\)\nb-algebra of a Fell bundle over
an étale, Hausdorff groupoid with unit space~\(\widecheck{A}\)
(Theorem~\ref{the:Cartan_Hausdorff}).  In particular, this result
explains how Renault's theory is a special case of Exel's theory.
As in Renault's theory, the groupoid and the Fell bundle are unique
up to isomorphism.

This leads to our second main result
(Theorem~\ref{thm:uniqueness_purely_outer_action}), which settles
the uniqueness of the crossed product decomposition.  This is more
subtle than uniqueness of the underlying twisted groupoid for
commutative Cartan subalgebras because many inverse semigroup
actions on a space give the same transformation groupoid.
Therefore, for an arbitrary action of a unital inverse
semigroup~\(S\) on a \(\Cst\)\nb-algebra~\(A\), we define a
\emph{refined action} that has the same crossed product and the same
dual groupoid \(\dual{A}\rtimes S\).  Two inverse semigroup actions
that model the same noncommutative Cartan subalgebra have isomorphic
refinements.

Our characterisations of noncommutative Cartan subalgebras have a
number of important consequences when combined with our previous
work on aperiodic inclusions in
\cites{Kwasniewski-Meyer:Aperiodicity, Kwasniewski-Meyer:Essential}.
On the one hand, there are noncommutative Cartan
\(\Cst\)\nb-subalgebras that do not detect ideals in~\(B\).  There
is a counterexample where~\(A\) is an AF-algebra and
\(B=A\rtimes_\alpha \Z\) for an automorphism~\(\alpha\) of~\(A\).
On the other hand, if~\(A\) is prime or contains an essential ideal
of Type~I, then every noncommutative Cartan inclusion
\(A\subseteq B\) is aperiodic in the sense
of~\cite{Kwasniewski-Meyer:Essential}.  Then it follows that~\(A\)
detects ideals in~\(B\) and even supports positive elements
in~\(B^+\); these properties are crucial to study the ideal
structure and pure infiniteness of~\(B\).  We prove elegant
characterisations of regular aperiodic \(\Cst\)\nb-inclusions
\(A\subseteq B\) with a faithful conditional expectation, either as
crossed products by aperiodic, closed inverse semigroup actions
on~\(A\), or as regular inclusions \(A\subseteq B\) whose dual
groupoids are effective and have closed space of units (Theorems
\ref{thm:characterisation_of_aperiodic_crossed_products}
and~\ref{thm:characterisation_of_aperiodic_crossed_products_via_topological_freeness}).

This article is organised as follows.  Section~\ref{sec:prelim} recalls
some general results on regular inclusions and inverse semigroup
actions and their crossed products.  We also define generalised
Fourier coefficients for inverse semigroup crossed products,
extending a similar definition for Fell bundles over groups.  In
Section~\ref{subsec:preserve_ce_regular}, we characterise when a
conditional expectation \(E\colon B\to A\) on an \(S\)\nb-graded
\(\Cst\)\nb-algebra induces the canonical expectation
on~\(A\rtimes S\).  This gives a sufficient condition, still rather
shallow, for an isomorphism \(A\rtimes_\red S \cong B\).
Section~\ref{sec:Exel} introduces noncommutative Cartan subalgebras
and characterises them in different ways.  The uniqueness of the
inverse semigroup action induced by a noncommutative Cartan
subalgebra is formulated in Section~\ref{sec:uniqueness_Cartan}.
This replaces the well known fact that Renault's twisted groupoid is
uniquely determined up to isomorphism.
Section~\ref{sec:pure_outer_vs_aperiodic} compares noncommutative
Cartan subalgebras with aperiodic inclusions.
Section~\ref{sec:Hausdorff_primitive_ideal_space} specialises to the
case when~\(\widecheck{A}\) is Hausdorff, with Sections
\ref{sec:Cartan_simple} and~\ref{sec:Cartan_commutative} treating
simple and commutative Cartan subalgebras.

\section{Preliminaries on inverse semigroup actions and crossed
  products}
\label{sec:prelim}

In this section, we briefly discuss some basic objects and facts
needed later.  See \cite{Kwasniewski-Meyer:Essential}*{Sections 1
  and~2} and the sources cited there for more details.  The only new
tool introduced here are the generalised Fourier coefficients in
Proposition~\ref{pro:harmonic_projections}.

\subsection{Inverse semigroup actions and regular inclusions}
\label{sec:isg_crossed}

Throughout this article, \(S\) stands for a unital inverse
semigroup with unit \(1\in S\).  It is equipped with the standard
partial order, which is defined by \(t \le u\) for \(t,u \in S\) if
and only if \(t = u t^* t\).

\begin{definition}[\cite{Kwasniewski-Meyer:Stone_duality}*{Definition~6.15}]
  \label{def:isg_grading}
  An \emph{\(S\)\nb-graded \(\Cst\)\nb-algebra} is a
  \(\Cst\)\nb-algebra~\(B\) with closed subspaces \(B_t\subseteq B\)
  for \(t\in S\) such that \(\sum_{t\in S} B_t\) is dense in~\(B\),
  \(B_t B_u \subseteq B_{t u}\) and \(B_t^* = B_{t^*}\) for all
  \(t,u\in S\), and \(B_t \subseteq B_u\) if \(t\le u\) in~\(S\).
  The grading is \emph{saturated} if \(B_t\cdot B_u = B_{t u}\) for
  all \(t,u\in S\).  We call \(A\defeq B_1\subseteq B\) the
  \emph{unit fibre} of the \(S\)\nb-grading.
\end{definition}

\begin{remark}
  If the \(S\)\nb-grading \((B_t)_{t\in S}\) is saturated, then the
  condition \(B_t \subseteq B_u\) if \(t\le u\) in~\(S\) follows
  from the other conditions.
\end{remark}

An \(S\)\nb-grading on a \(\Cst\)\nb-algebra~\(B\) gives a family of
Banach spaces~\((B_t)_{t\in S}\) with conjugate-linear isometries
\(B_t \congto B_{t^*}\), \(x\mapsto x^*\), multiplication maps
\(B_t \times B_u \to B_{t\cdot u}\) for all \(t,u\in S\), and
isometric embeddings \(B_t \hookrightarrow B_u\) for all
\(t,u\in S\) with \(t \le u\).  This is the data of a Fell bundle
over~\(S\).  It is a \emph{saturated Fell bundle} if the
multiplication maps are surjective.  The definition of a Fell bundle
imposes a long list of conditions on this data
(see~\cite{Exel:noncomm.cartan}).  These conditions hold if and only
if the data is realised by some \(S\)\nb-graded \(\Cst\)\nb-algebra,
which may be taken saturated if the Fell bundle is saturated.  The
definition of a Fell bundle simplifies in the saturated case.  Namely,
saturated Fell bundles are equivalent to the following
\(S\)\nb-actions by Hilbert bimodules:

\begin{definition}[\cite{Buss-Meyer:Actions_groupoids}]
  \label{def:S_action_Cstar}
  An \emph{action} of~\(S\) on a \(\Cst\)\nb-algebra~\(A\) (by
  Hilbert bimodules) consists of Hilbert
  \(A\)\nb-bimodules~\(\Hilm_t\) for \(t\in S\) and Hilbert bimodule
  isomorphisms
  \(\mu_{t,u}\colon \Hilm_t\otimes_A \Hilm_u\congto \Hilm_{tu}\) for
  \(t,u\in S\), such that
  \begin{enumerate}[label=\textup{(A\arabic*)}]
  \item \label{enum:AHB3}%
    for all \(t,u,v\in S\), the following diagram commutes
    (associativity):
    \[
    \begin{tikzpicture}[baseline=(current bounding box.west)]
      \node (1) at (0,1) {\((\Hilm_t\otimes_A \Hilm_u) \otimes_A \Hilm_v\)};
      \node (1a) at (0,0) {\(\Hilm_t\otimes_A (\Hilm_u \otimes_A \Hilm_v)\)};
      \node (2) at (5,1) {\(\Hilm_{tu} \otimes_A \Hilm_v\)};
      \node (3) at (5,0) {\(\Hilm_t\otimes_A \Hilm_{uv}\)};
      \node (4) at (7,.5) {\(\Hilm_{tuv}\)};
      \draw[<->] (1) -- node[swap] {associator}    (1a);
      \draw[cdar] (1) -- node {\(\mu_{t,u}\otimes_A \Id_{\Hilm_v}\)} (2);
      \draw[cdar] (1a) -- node[swap] {\(\Id_{\Hilm_t}\otimes_A\mu_{u,v}\)} (3);
      \draw[cdar] (3.east) -- node[swap] {\(\mu_{t,uv}\)} (4);
      \draw[cdar] (2.east) -- node {\(\mu_{tu,v}\)} (4);
    \end{tikzpicture}
    \]
  \item \label{enum:AHB1}%
    \(\Hilm_1\) is the identity Hilbert \(A,A\)-bimodule~\(A\);
  \item \label{enum:AHB2}%
    \(\mu_{t,1}\colon \Hilm_t\otimes_A A\congto \Hilm_t\) and
    \(\mu_{1,t}\colon A\otimes_A \Hilm_t\congto \Hilm_t\) for
    \(t\in S\) are the maps defined by
    \(\mu_{1,t}(a\otimes\xi)=a\cdot\xi\) and
    \(\mu_{t,1}(\xi\otimes a) = \xi\cdot a\) for \(a\in A\),
    \(\xi\in\Hilm_t\).
  \end{enumerate}
\end{definition}

Any \(S\)\nb-action by Hilbert bimodules comes with canonical
involutions \(\Hilm_t^*\to \Hilm_{t^*}\), \(x\mapsto x^*\), and
inclusion maps \(j_{u,t}\colon \Hilm_t\to\Hilm_u\) for \(t \le u\)
that satisfy the conditions required for a saturated Fell bundle
in~\cite{Exel:noncomm.cartan} (see
\cite{Buss-Meyer:Actions_groupoids}*{Theorem~4.8}).  Thus
\(S\)\nb-actions by Hilbert bimodules are equivalent to saturated
Fell bundles over~\(S\).

\begin{definition}
  \label{def:normaliser}
  Let \(A\subseteq B\) be a \(\Cst\)\nb-subalgebra.  We call the
  elements of
  \[
  N(A,B)\defeq \setgiven{b\in B}{b A b^*\subseteq A,\ b^* A b\subseteq A}
  \]
  \emph{normalisers} of~\(A\) in~\(B\)
  (see~\cite{Kumjian:Diagonals}).  We call the inclusion
  \(A\subseteq B\) \emph{regular} if it is non-degenerate and the
  linear span of~\(N(A,B)\) is dense in~\(B\)
  (see~\cite{Renault:Cartan.Subalgebras}).
\end{definition}

If \(b\in N(A,B)\) and \(a_1,a_2\in A\), then
\(a_1 \cdot b \cdot a_2 \in N(A,B)\).  It often happens that
\(b,c\in N(A,B)\), but \(b+c\notin N(A,B)\).  So \(N(A,B)\) is not a
vector subspace.

\begin{proposition}[\cite{Kwasniewski-Meyer:Aperiodicity}*{Proposition~2.11},
  \cite{Kwasniewski-Meyer:Stone_duality}*{Proposition~6.26}]
  \label{prop:regular_vs_inverse_semigroups}
  The following are equivalent for a \(\Cst\)\nb-inclusion
  \(A\subseteq B\):
  \begin{enumerate}
  \item \label{enu:non_commutative_cartans1}%
    \(A\) is a regular subalgebra of~\(B\);
  \item \label{enu:non_commutative_cartans2}%
    \(A\) is the unit fibre for some inverse semigroup grading
    on~\(B\);
  \item \label{enu:non_commutative_cartans3}%
    \(A\) is the unit fibre for some saturated inverse semigroup
    grading on~\(B\).
  \end{enumerate}
  If these equivalent conditions hold, then
  \[
  \Slice(A,B)\defeq \setgiven{M \subseteq N(A,B)}
  {M \text{ is a closed linear subspace, } AM\subseteq M,\ MA\subseteq M}
  \]
  with the operations
  \(M\cdot N\defeq \operatorname{\overline{span}} {}\setgiven{m n}{m
    \in M,\ n\in N}\) and \(M^*\defeq \setgiven{m^*}{m \in M}\) is
  an inverse semigroup.  And the subspaces \(M\in \Slice(A,B)\) form
  a saturated \(\Slice(A,B)\)-grading on~\(B\).  Let~\(S\) be a
  unital inverse semigroup and \((B_t)_{t\in S}\) an
  \(S\)\nb-grading on~\(B\) with \(A=B_1\).  Then
  \(B_t\in\Slice(A,B)\) for all \(t\in S\).  The map
  \(t\mapsto B_t\) is a homomorphism \(S\to \Slice(A,B)\) if and
  only if the \(S\)\nb-grading on~\(B\) is saturated.
\end{proposition}

\begin{definition}
  Elements of~\(\Slice(A,B)\) are called \emph{slices}.
\end{definition}

\subsection{Full and reduced crossed products for inverse semigroup
  actions}
\label{sec:crossed_isg}

Let \(\Hilm=\bigl((\Hilm_t)_{t\in S}, (\mu_{t,u})_{t,u\in S}\bigr)\)
be an \(S\)\nb-action on a \(\Cst\)\nb-algebra~\(A\).  For any
\(t\in S\), let \(\rg(\Hilm_t)\) and~\(\s(\Hilm_t)\) be the ideals
in~\(A\) generated by the left and right inner products of vectors
in~\(\Hilm_t\), respectively.  Thus~\(\Hilm_t\) is an
\(\rg(\Hilm_t)\)-\(\s(\Hilm_t)\)-imprimitivity bimodule.  And
\(\s(\Hilm_t) = \s(\Hilm_{t^*t}) = \rg(\Hilm_{t^*t}) =
\rg(\Hilm_{t^*})\).  If \(v\le t\), then the inclusion
map~\(j_{t,v}\) restricts to a Hilbert bimodule isomorphism
\(\Hilm_v \congto \rg(\Hilm_v) \cdot \Hilm_t = \Hilm_t\cdot
\s(\Hilm_v)\).  For \(t,u\in S\) and \(v\le t,u\), this gives
Hilbert bimodule isomorphisms
\( \vartheta^v_{u,t}\colon \Hilm_t\cdot \s(\Hilm_v)
\xleftarrow[\cong]{j_{t,v}} \Hilm_v \xrightarrow[\cong]{j_{u,v}}
\Hilm_u\cdot \s(\Hilm_v)\).  Let
\begin{equation}
  \label{eq:Itu}
  I_{t,u} \defeq \overline{\sum_{v \le t,u} \s(\Hilm_v)}
\end{equation}
be the closed ideal generated by~\(\s(\Hilm_v)\) for \(v\le t,u\).
This is contained in \(\s(\Hilm_t)\cap\s(\Hilm_u)\), and the
inclusion may be strict.  There is a unique Hilbert bimodule
isomorphism
\begin{equation}
  \label{eq:Def-thetas}
  \vartheta_{u,t}\colon \Hilm_t\cdot I_{t,u}
  \congto \Hilm_u\cdot I_{t,u}
\end{equation}
that restricts to~\(\vartheta_{u,t}^v\) on
\(\Hilm_t\cdot \s(\Hilm_v)\) for all \(v\le t,u\) by
\cite{Buss-Exel-Meyer:Reduced}*{Lemma~2.5}.

The \emph{algebraic crossed product} \(A\rtimes_\alg S\) is the
quotient vector space of \(\bigoplus_{t\in S} \Hilm_t\) by the
linear span of \(\vartheta_{u,t}(\xi)\delta_u-\xi\delta_t\) for all
\(t,u\in S\) and \(\xi\in\Hilm_t\cdot I_{t,u}\).  It is a
\Star{}algebra with multiplication and involution induced by the
maps~\(\mu_{t,u}\) and the involutions
\(\Hilm_t^* \to \Hilm_{t^*}\).  There is a maximal
\(\Cst\)\nb-norm on \(A\rtimes_\alg S\).

\begin{definition}
  The \emph{\textup{(}full\textup{)} crossed product} \(A\rtimes S\)
  of the action~\(\Hilm\) is the maximal \(\Cst\)\nb-completion of
  the \Star{}algebra \(A\rtimes_\alg S\).
\end{definition}

\begin{remark}
  \label{rem:action_vs_grading}
  The Hilbert \(A\)-bimodules \((\Hilm_t)_{t\in S}\) embed naturally
  into \(A\rtimes S\), and then they form a saturated
  \(S\)\nb-grading of \(A\rtimes S\).  In particular,
  \(A\subseteq A\rtimes S\) is a regular \(\Cst\)\nb-inclusion and
  the subspaces \((\Hilm_t)_{t\in S}\) form an inverse subsemigroup
  of \(\Slice(A,A\rtimes S)\).  For every \(S\)\nb-graded
  \(\Cst\)\nb-algebra~\(B\) with grading \((\Hilm_t)_{t\in S}\),
  there is a surjective \Star{}homomorphism \(A\rtimes S\to B\)
  which is the identity map on the fibres of the grading.  This
  universal property determines \(A\rtimes S\) uniquely up to
  isomorphism.
\end{remark}

The reduced section \(\Cst\)\nb-algebra of a Fell bundle over~\(S\)
is defined first in~\cite{Exel:noncomm.cartan}.  An equivalent
definition appears in~\cite{Buss-Exel-Meyer:Reduced}, where it is
called the reduced crossed product \(A\rtimes_\red S\) of the
action~\(\Hilm\).  Here we define \(A\rtimes_\red S\) in another
equivalent way as a quotient of \(A\rtimes S\), using the canonical
weak conditional expectation introduced
in~\cite{Buss-Exel-Meyer:Reduced} and studied in
\cite{Kwasniewski-Meyer:Essential}.  A \emph{generalised
  expectation} for a \(\Cst\)\nb-inclusion \(A\subseteq B\)
consists of another \(\Cst\)\nb-inclusion \(A\subseteq \tilde{A}\)
and a completely positive, contractive map
\(E\colon B \to \tilde{A}\) that restricts to the identity map
on~\(A\).  Then~\(E\) must be an \(A\)\nb-bimodule map (see
\cite{Kwasniewski-Meyer:Essential}*{Lemma~3.2}).  If~\(\tilde{A}\)
is the bidual von Neumann algebra~\(A''\), then~\(E\) is called a
\emph{weak conditional expectation}.  We recall that a generalised
expectation \(E\colon B\to \tilde{A} \supseteq A\) is
\emph{faithful} if \(E(b^* b)=0\) for some \(b\in B\) implies
\(b=0\).  It is \emph{almost faithful} if \(E((b c)^* b c)=0\) for
all \(c\in B\) and some \(b\in B\) implies \(b=0\) (equivalently,
the largest closed two-sided ideal contained in \(\ker E\) is zero,
see \cite{Kwasniewski:Exel_crossed}*{Proposition~2.2} or
\cite{Kwasniewski-Meyer:Essential}*{Proposition~3.5}).  There are
natural examples of almost faithful conditional expectations that
are not faithful (see
\cite{Brownlowe-Raeburn:Exel_Cuntz-Pimsner}*{Example~4.6}).  An
almost faithful generalised expectation~\(E\) is faithful if and
only if it \emph{symmetric}, that is, \(E(b^* b)=0\) is equivalent
to \(E(b b^*)=0\) for all \(b\in B\) (see
\cite{Kwasniewski-Meyer:Essential}*{Corollary~3.7}).

\begin{proposition}
  \label{prop:fast_intro_for_reduced}
  There is a canonical weak conditional expectation
  \(E\colon A\rtimes S \to A''\) defined through the formula
  \begin{equation}
    \label{eq:formula-cond.exp}
    E(\xi\delta_t)=\stlim_i\vartheta_{1,t}(\xi\cdot u_i)
  \end{equation}
  for \(\xi\in \Hilm_t\) and \(t\in S\), where~\((u_i)\) is an
  approximate unit for~\(I_{1,t}\) and \(\stlim\) denotes the limit
  in the strict topology on \(\Mult(I_{1,t})\subseteq A''\).
  The largest ideal contained in \(\ker E\) is
  \[
    \Null_E\defeq  \setgiven{b\in A\rtimes S}{E(b^* b)=0}
    = \setgiven{b\in A\rtimes S}{E(b b^*)=0},
  \]
  and~\(E\) factors through to a faithful weak conditional
  expectation \((A\rtimes S)/\Null_E \to A''\).
\end{proposition}

\begin{proof}
  The first part follows from
  \cite{Buss-Exel-Meyer:Reduced}*{Lemma~4.5} (see also
  \cite{Kwasniewski-Meyer:Essential}*{Proposition~3.15}).  The
  second part follows from
  \cite{Kwasniewski-Meyer:Essential}*{Proposition~3.15 and
    Theorem~3.20}.
\end{proof}

\begin{definition}
  \label{def:reduced_crossed}
  The \emph{reduced crossed product} of the inverse semigroup
  action~\(\Hilm\) is the quotient
  \(A\rtimes_\red S\defeq (A\rtimes S)/\Null_E\).  So
  \(A\rtimes_\red S\) is the unique quotient of \(A\rtimes S\) for
  which the weak conditional expectation
  in~\eqref{eq:formula-cond.exp} factors through to a faithful weak
  conditional expectation
  \(E_\red\colon A\rtimes_\red S \to A''\).
\end{definition}

\begin{remark}
  \label{rem:cross_product_graded}
  The canonical maps from \(A\rtimes_\alg S\) to \(A\rtimes S\) and
  to \(A\rtimes_\red S\) are injective by
  \cite{Buss-Exel-Meyer:Reduced}*{Proposition~4.3}.  Both
  \(A\rtimes S\) and \(A\rtimes_\red S\) are naturally
  \(S\)\nb-graded with the same Fell bundle \((\Hilm_t)_{t\in S}\)
  over~\(S\).
\end{remark}

\begin{example}[Fell bundles over groupoids]
  \label{ex:Fell_bundles_over_groupoids}
  Let \(\A=(A_\gamma)_{\gamma\in H}\) be an upper semicontinuous
  Fell bundle over an étale groupoid~\(H\) with locally compact and
  Hausdorff unit space~\(X\).  Let
  \[
    \Bis(H)\defeq
    \setgiven{U\subseteq H}{U\text{ is open and }s,r\colon U\to X
      \text{ are injective}}
  \]
  be the set of \emph{bisections} of~\(H\).  It is a unital inverse
  semigroup, with multiplication and involution inherited from~\(H\)
  and with the unit~\(X\).  If \(U\in\Bis(H)\), let~\(A_U\) be the
  space of \(\Cont_0\)\nb-sections of the restriction
  of~\((A_\gamma)_{\gamma\in H}\) to~\(U\).  The spaces~\(A_U\) for
  \(U\in \Bis(H)\) form a Fell bundle over \(\Bis(H)\) in a natural
  way (see \cites{BussExel:Fell.Bundle.and.Twisted.Groupoids,
    Kwasniewski-Meyer:Essential}).  In particular, they are Hilbert
  bimodules over the \(\Cont_0(X)\)\nb-algebra \(A\defeq A_X\)
  corresponding to the bundle of \(\Cst\)\nb-algebras
  \((A_x)_{x\in X}\).  Let \(S\subseteq \Bis(H)\) be a unital
  inverse subsemigroup which is \emph{wide}
  in the sense that \(\bigcup S = H\) and \(U\cap V\) is a union of
  bisections in~\(S\) for all \(U,V\in S\).  Then the full and
  reduced \(\Cst\)\nb-algebras for the Fell bundles
  \((A_U)_{U\in S}\) and \((A_\gamma)_{\gamma\in H}\) are naturally
  isomorphic (see
  \cite{BussExel:Fell.Bundle.and.Twisted.Groupoids}*{Theorem~2.13},
  \cite{Buss-Meyer:Actions_groupoids}*{Corollary~5.6} and
  \cite{Buss-Exel-Meyer:Reduced}*{Theorem~8.11}).  The results in
  \cites{Buss-Meyer:Actions_groupoids, Buss-Exel-Meyer:Reduced} are
  formulated for saturated Fell bundles, but they remain true for
  non-saturated ones, with the same proofs (see
  also~\cite{Kwasniewski-Meyer:Essential}).  Thus if
  \((A_U)_{U\in S}\) is saturated --~which is always the case
  when~\((A_\gamma)_{\gamma\in H}\) is saturated~-- then
  \(A\rtimes S\cong \Cst(H,\A)\) and
  \(A\rtimes_\red S\cong \Cst_\red(H,\A)\).  If \((A_U)_{U\in S}\)
  is not saturated, we extend it to a saturated Fell bundle over an
  inverse semigroup~\(\tilde{S}\) consisting of all Hilbert
  bimodules of the form \(A_U J\), where \(U\in S\) and~\(J\) is an
  ideal in~\(A\).  Then \(A\rtimes \tilde{S}\cong \Cst(H,\A)\) and
  \(A\rtimes_\red \tilde{S}\cong\Cst_\red(H,\A)\) (see
  \cite{Kwasniewski-Meyer:Essential}*{Propositions 7.6 and~7.9}).
\end{example}

Fell line bundles over~\(H\) correspond to ``twists'' of~\(H\).  The
resulting section \(\Cst\)\nb-algebras are full and reduced
\emph{twisted} groupoid \(\Cst\)\nb-algebras.

\subsection{Dual groupoids and closed actions}
\label{sec:dual_groupoids}

We briefly recall how inverse semigroup actions are related to étale
groupoids.  Let \(\haction_t\colon \D_t\to \D_{t^*}\) for \(t\in S\)
be partial homeomorphism forming an action of~\(S\) on a topological
space~\(X\).  The arrows of the transformation groupoid
\(X\rtimes S\) are equivalence classes of pairs~\((t,x)\) for
\(x\in \D_t\subseteq X\), where \((t,x)\) and~\((t',x')\)
are equivalent if \(x=x'\) and there is \(v\in S\) with
\(v\le t, t'\) and \(x\in \D_v\).  The range and source maps
\(\rg,\s\colon X\rtimes S \rightrightarrows X\) and the
multiplication are defined by \(\rg([t,x])\defeq \haction_t(x)\),
\(\s([t,x])\defeq x\), and
\([t,\haction_u(x)] \cdot [u,x] = [t\cdot u,x]\).  And the topology
on~\(X\rtimes S\) is such that \([t,x]\mapsto x\) is a homeomorphism
from an open subset of~\(X\rtimes S\) onto~\(\D_t\) for each
\(t\in S\).  Thus~\(X\rtimes S\) is an étale topological groupoid.

Every étale topological groupoid~\(H\) arises in this way.  Namely,
if \(S\subseteq \Bis(H)\) is a unital and wide inverse subsemigroup,
then~\(H\) is naturally isomorphic to the transformation
groupoid~\(X\rtimes S\) of the associated action (see
\cite{Exel:Inverse_combinatorial}*{Propositions 5.3 and~5.4} or
\cite{Kwasniewski-Meyer:Essential}*{Proposition 2.2}).

\begin{remark}
  \label{rem:fell_bundles_over_groupoids3}
  The transformation groupoid described above differs, in general,
  from the groupoid of germs considered
  in~\cite{Renault:Cartan.Subalgebras}.  By
  \cite{Renault:Cartan.Subalgebras}*{Proposition~3.2}, a
  groupoid~\(H\) is isomorphic to the groupoid of germs of its
  bisections if and only if it is effective.
\end{remark}

Let~\(A\) be a \(\Cst\)\nb-algebra with an action~\(\Hilm\) of a
unital inverse semigroup~\(S\).  Let \(\dual{A}\) and
\(\widecheck{A}=\Prim(A)\) be the space of irreducible
representations and the primitive ideal space of~\(A\),
respectively.  Open subsets in~\(\dual{A}\) and in~\(\widecheck{A}\)
are in natural bijection with ideals in~\(A\).  The action of~\(S\)
on~\(A\) induces actions \(\dual{\Hilm}= (\dual{\Hilm}_t)_{t\in S}\)
and \(\widecheck{\Hilm}=(\widecheck{\Hilm}_t)_{t\in S}\) of~\(S\) by
partial homeomorphisms on \(\dual{A}\) and~\(\widecheck{A}\),
respectively (see \cite{Buss-Meyer:Actions_groupoids}*{Lemma~6.12},
\cite{Kwasniewski-Meyer:Essential}*{Section~2.3}).  The
homeomorphisms
\(\widecheck{\Hilm}_t\colon \widecheck{\s(\Hilm_t)} \congto
\widecheck{\rg(\Hilm_t)}\) and
\(\dual{\Hilm}_t\colon \dual{\s(\Hilm_t)} \congto
\dual{\rg(\Hilm_t)}\) are given by Rieffel's correspondence and
induction of representations, respectively.  Namely,
\(\dual{\s(\Hilm_t)}\) consists of all \(\pi\in\dual{A}\) that are
non-degenerate on \(\s(\Hilm_t) = \braket{\Hilm_t}{\Hilm_t}\), and
for \(\pi\in\dual{\s(\Hilm_t)}\), we let \(\dual{\Hilm_t}(\pi)\) be
the left multiplication action of~\(A\) on the tensor product
\(\Hilm_t \otimes_A \Hils_\pi\).  And
\(\widecheck{\s(\Hilm_t)} = \setgiven{\ker
  \pi}{\pi\in\dual{\s(\Hilm_t)}}\) consists of all primitive ideals
in \(A\) that do not contain \(\s(\Hilm_t)\) and
\(\widecheck{\Hilm}_t(\ker \pi)=\ker\widehat{\Hilm}_t(\pi)\) for
every \(\pi\in\dual{\s(\Hilm_t)}\).

\begin{definition}[\cites{Kwasniewski-Meyer:Stone_duality, Kwasniewski-Meyer:Essential}]
  \label{def:dual_groupoid}
  We call \(\dual{\Hilm}\) and~\(\widecheck{\Hilm}\) \emph{dual
    actions} to the action~\(\Hilm\) of~\(S\) on~\(A\).  The
  transformation groupoids \(\dual{A}\rtimes S\)
  and~\(\widecheck{A}\rtimes S\) are called \emph{dual groupoids}
  of~\(\Hilm\).
\end{definition}

\begin{example}
  \label{ex:fell_bundles_over_groupoids3}
  Let~\(A\) be a commutative \(\Cst\)\nb-algebra.  So
  \(A\cong \Cont_0(X)\) and
  \(\dual{A} \cong \widecheck{A} \cong X\).  Consider an inverse
  semigroup action on~\(A\).  The corresponding saturated Fell
  bundles are studied
  in~\cite{BussExel:Fell.Bundle.and.Twisted.Groupoids}, where they
  are called \emph{semi-Abelian}.  As shown
  in~\cite{BussExel:Fell.Bundle.and.Twisted.Groupoids}, they are
  equivalent to Fell line bundles over étale groupoids and to
  twisted étale groupoids.  The relevant groupoid is the dual
  groupoid of the action, \(H\defeq X\rtimes S\).  There is a unique
  twist~\(\Sigma\) over~\(H\) with an isomorphism
  \(\Cst(H,\Sigma)\cong A\rtimes S\), which descends to an
  isomorphism \(\Cst_\red(H,\Sigma)\cong A\rtimes_\red S\) (see
  also Example~\ref{ex:Fell_bundles_over_groupoids}).
\end{example}

\begin{definition}
  \label{def:Hausdorffness}
  An inverse semigroup action~\(\Hilm\) on a
  \(\Cst\)\nb-algebra~\(A\) is called \emph{closed} if the weak
  conditional expectation \(E\colon A\rtimes S\to A''\) given by
  \eqref{eq:formula-cond.exp} is \(A\)\nb-valued, that is, it is a
  genuine conditional expectation
  \(A\rtimes S\to A\subseteq A\rtimes S\).
\end{definition}

\begin{proposition}
  \label{prop:conditional_expectation}
  Let~\(\Hilm\) be an action of~\(S\) on~\(A\).  The following are
  equivalent:
  \begin{enumerate}
  \item \label{prop:conditional_expectation1}%
    the action~\(\Hilm\) is closed, that is,
    \(E(A\rtimes S)\subseteq A\);
  \item \label{prop:conditional_expectation2}%
    the subset of units~\(\widecheck{A}\) is closed in the arrow
    space \(\widecheck{A}\rtimes S\);
  \item \label{prop:conditional_expectation3}%
    the subset of units~\(\dual{A}\) is closed in the arrow space
    \(\dual{A}\rtimes S\);
  \item \label{prop:conditional_expectation4}%
    for each \(t\in S\), the ideal~\(I_{1,t}\) defined
    in~\eqref{eq:Itu} is complemented in~\(\s(\Hilm_t)\).
  \end{enumerate}
  Let \(t\in S\) and let
  \(I_{1,t}^\bot \defeq \setgiven{a\in A}{a\cdot I_{1,t}=0}\) be the
  annihilator of~\(I_{1,t}\).  If
  \ref{prop:conditional_expectation1}--\ref{prop:conditional_expectation4}
  hold, then~\(E|_{\Hilm_t}\) is the projection onto the first
  summand in the decomposition
  \begin{equation}
    \label{eq:decompose_Hilm_closed}
    \Hilm_t = \Hilm_t \cdot I_{1,t} \oplus \Hilm_t \cdot I_{1,t}^\bot
    \xrightarrow[\cong]{\vartheta_{1,t} \oplus \Id}
    I_{1,t} \oplus \Hilm_t \cdot I_{1,t}^\bot.
  \end{equation}
\end{proposition}

\begin{proof}
  Combine \cite{Buss-Exel-Meyer:Reduced}*{Theorem~6.5 and
    Proposition~6.3} and
  \cite{Kwasniewski-Meyer:Essential}*{Proposition~3.18}.
\end{proof}

\subsection{Generalised Fourier coefficients for closed actions}
\label{sec:Fourier}

Let~\(G\) be a group and let~\(B\) be a \(\Cst\)\nb-algebra with a
topological \(G\)\nb-grading \((B_t)_{t\in G}\).  Besides the
conditional expectation \(E\colon B\to B_1\), there are projections
\(E_t\colon B\to B_t\) for all \(t\in G\), with \(E = E_1\) (see
\cite{Exel:Partial_dynamical}*{Corollary~19.6}).  We now define
analogous maps for a closed action of an inverse semigroup~\(S\) by
Hilbert bimodules.  We cannot expect these to be defined on all
of~\(A\rtimes_\red S\) because this is impossible in very easy
examples: let \(S=\{0,1\}\), \(A=B = B_1\), and let~\(B_0\) be an
ideal in~\(A\) that is not complemented.

\begin{proposition}
  \label{pro:harmonic_projections}
  Let~\(S\) be a unital inverse semigroup and let
  \((\Hilm_t,\mu_{t,u})_{t,u\in S}\) be a closed action of~\(S\) on
  a \(\Cst\)\nb-algebra~\(A\) by Hilbert bimodules.  For each
  \(t\in S\) there is a unique contractive linear map
  \[
    E_t\colon \rg(\Hilm_t)\cdot (A\rtimes_\red S)
    \onto \Hilm_t\subseteq \rg(\Hilm_t)\cdot (A\rtimes_\red S)
  \]
  that satisfies \(x^* E_t(y) = E(x^* y)\) for all \(x\in \Hilm_t\),
  \(y\in \rg(\Hilm_t)\cdot (A\rtimes_\red S)\).  The map~\(E_t\) is
  idempotent and \(A\)\nb-bilinear.  The map~\(E_1\) for the
  unit \(1\in S\) is the canonical conditional expectation~\(E\).
  Let \(x\in A\rtimes_\red S\).  Then \(x=0\) if and only if
  \(E_t(a x)=0\) for all \(t\in S\) and all \(a\in \rg(\Hilm_t)\).
\end{proposition}

\begin{proof}
  The conditional expectation~\(E\) defines an \(A\)\nb-valued inner
  product \(\braket{x}{y} \defeq E(x^* y)\) on~\(A\rtimes_\red S\).
  Let \(\ell^2(S,A)\) denote the resulting Hilbert \(A\)\nb-module
  completion of \(A\rtimes_\red S\).  If \(x\in \Hilm_t\),
  \(y\in \rg(\Hilm_t)\cdot (A\rtimes_\red S)\), then
  \(\braket{x}{E_t(y)} = x^* E_t(y)\) and
  \(\braket{x}{y} = E(x^* y)\).  So the formula defining~\(E_t\)
  says exactly that~\(E_t\) should be the orthogonal projection
  onto~\(\Hilm_t\) in \(\Bound(\rg(\Hilm_t) \cdot \ell^2(S,A))\),
  restricted to \(\rg(\Hilm_t)\cdot(A\rtimes_\red S)\).  To
  construct~\(E_t\), it remains to show that~\(\Hilm_t\) is
  complemented as a Hilbert \(A\)\nb-submodule of
  \(\rg(\Hilm_t)\cdot \ell^2(S,A)\).

  The Hilbert submodules \(\rg(\Hilm_t)\cdot \Hilm_u\) for
  \(u\in S\) are linearly dense in
  \(\rg(\Hilm_t)\cdot(A\rtimes_\red S)\).  The orthogonal
  decomposition
  \(\Hilm_{t^* u} = I_{1,t^* u} \oplus \Hilm_{t^* u}\cdot I_{1,t^*
    u}^\bot\) in~\eqref{eq:decompose_Hilm_closed} implies an
  orthogonal decomposition
  \begin{equation}
    \label{eq:decompose_Bttstaru}
    \rg(\Hilm_t)\cdot \Hilm_u
    = \Hilm_t \cdot \Hilm_t^* \cdot \Hilm_u
    = \Hilm_t \cdot \Hilm_{t^* u}
    = \Hilm_t\cdot I_{1,t^* u} \oplus
    \Hilm_t \cdot \Hilm_{t^* u}\cdot I_{1,t^* u}^\bot.
  \end{equation}
  The summand \(\Hilm_t\cdot I_{1,t^* u}\) is contained
  in~\(\Hilm_t\).  In fact, it is equal to
  \(\Hilm_t\cdot I_{t, u}= \Hilm_t\cap \Hilm_u= \Hilm_t\cap
  (\rg(\Hilm_t)\Hilm_u)\), where the intersection is taken in
  \(A\rtimes_\red S\).  We claim that the other summand is
  orthogonal to~\(\Hilm_t\).  Indeed, if \(x_t,y_t\in \Hilm_t\),
  \(z_{t^* u}\in \Hilm_{t^* u}\cdot I_{1,t^* u}^\bot\), then
  \[
  \braket{y_t\cdot z_{t^* u}}{x_t}
  = E(z_{t^* u}^* \cdot y_t^*\cdot x_t)
  = E(z_{t^* u})^*\cdot y_t^*x_t
  = 0
  \]
  because~\(E\) vanishes on \(\Hilm_{t^* u}\cdot I_{1,t^* u}^\bot\)
  by~\eqref{eq:formula-cond.exp}.  So the orthogonal projection
  to~\(\Hilm_t\) exists on \(\rg(\Hilm_t)\cdot \Hilm_u\) and is the
  projection onto the first summand in the decomposition
  in~\eqref{eq:decompose_Bttstaru}.  Since the submodules
  \(\rg(\Hilm_t)\cdot \Hilm_u\) for \(u\in S\) are linearly dense in
  \(\rg(\Hilm_t)\cdot \ell^2(S,A)\), it follows that~\(\Hilm_t\) is
  complemented in \(\rg(\Hilm_t)\cdot \ell^2(S,A)\).

  The orthogonal projection to~\(\Hilm_t\) is contractive for the
  Hilbert \(A\)\nb-module norm on \(\rg(\Hilm_t)\cdot \ell^2(S,A)\),
  which is dominated by the \(\Cst\)\nb-norm on
  \(\rg(\Hilm_t)\cdot (A\rtimes_\red S)\).  Hence~\(E_t\) is
  contractive as asserted.  It is idempotent and right
  \(A\)\nb-linear by construction.  It is left \(A\)\nb-linear
  because the submodules in question are all invariant under left
  multiplication by~\(A\).  The uniqueness of~\(E_1\) and the
  \(A\)\nb-linearity of~\(E\) imply \(E_1=E\).

  We show the last claim in the proposition.  Let
  \(x\in A\rtimes_\red S\) satisfy \(E_t(a x)=0\) for all
  \(t\in S\), \(a\in \rg(\Hilm_t)\).  Let \(y\in \Hilm_t\).  The
  Cohen--Hewitt Factorisation Theorem gives \(a\in \rg(\Hilm_t)\)
  and \(y'\in \Hilm_t\) with \(y = a y'\).  We compute
  \(\braket{x}{y} = \braket{a^* x}{y'} = \braket{E_t(a^* x)}{y'} =
  0\).  Thus~\(x\) is orthogonal to~\(\Hilm_t\) for all \(t\in S\).
  This implies that~\(x\) is mapped to~\(0\) in \(\ell^2(S,A)\)
  because the linear span of~\(\Hilm_t\) for \(t\in S\) is dense.
  Then \(E_\red(x^* x) = \braket{x}{x} = 0\), and this implies
  \(x=0\) because~\(E_\red\) is faithful.
\end{proof}

\section{When the conditional expectation preserves the grading}
\label{subsec:preserve_ce_regular}

Before we tackle the main issues in Exel's theory of noncommutative
Cartan subalgebras, we prove a  more basic result.  Let~\(S\) be
a unital inverse semigroup and let~\(B\) be a \(\Cst\)\nb-algebra
with an \(S\)\nb-grading \((B_t)_{t\in S}\).  Let \(A\defeq B_1\) be
its unit fibre and view the grading as an action of~\(S\) on~\(A\)
by Hilbert bimodules.  The grading gives a representation of this
action, which induces a surjective \Star{}homomorphism
\(\pi\colon A\rtimes S \to B\).  In addition, let \(E\colon B\to A\)
be an almost faithful conditional expectation.  If
\(E\circ \pi\colon A\rtimes S\to A\) is equal to the canonical weak
conditional expectation on the crossed product,
\(E_0\colon A\rtimes S \to A''\), then~\(\pi\) descends to a
\Star{}isomorphism \(A\rtimes_\red S \congto B\) that intertwines
the canonical conditional expectations.  And then~\(E_0\) is
\(A\)\nb-valued, that is, the action has to be closed.

So we need conditions that are sufficient for \(E\circ \pi = E_0\).
We are going to show that a conditional expectation on~\(B\)
satisfies \(E\circ \pi = E_0\) if and only if the \(S\)\nb-grading
is ``wide'' and~\(E\) ``preserves'' it.  We will use this criterion
to characterise noncommutative Cartan subalgebras.

Our main theorems need a genuine conditional expectation
\(B\to A\).  In general, however, an inverse semigroup crossed
product only carries an \(A''\)\nb-valued conditional expectation.
Other interesting targets for conditional expectations are the
injective hull or the local multiplier algebra of~\(A\)
(see~\cite{Kwasniewski-Meyer:Essential}).  We prove some technical
lemmas for generalised conditional expectations taking values in an
arbitrary \(\Cst\)\nb-algebra~\(\tilde{A}\) containing~\(A\).  This
extra generality is not going to be used in this article.  We hope,
however, that it will prove useful for future generalisations of the
theory to non-closed actions.  The theory of ``weak'' Cartan
subalgebras for non-Hausdorff groupoids is far more difficult than
the usual theory of Cartan subalgebras and only in its infancy (see
\cites{Exel-Pitts:Weak_Cartan,
  Clark-Exel-Pardo-Sims-Starling:Simplicity_non-Hausdorff}).

\begin{definition}[compare
  \cite{Exel:noncomm.cartan}*{Definition~13.7}]
  \label{def:admissible_grading}
  An \(S\)\nb-grading \((B_t)_{t\in S}\) is \emph{wide} if, for all
  \(t\in S\), \(\sum_{v\le 1,t} B_v\) is dense in \(B_t \cap A\).
\end{definition}

\begin{example}
  The tautological grading of~\(B\) over the slice inverse semigroup
  \(\Slice(A,B)\) is always wide.  All gradings by groups are wide.
\end{example}

\begin{definition}
  \label{def:E_preserves_grading}
  A conditional expectation \(E\colon B\to A\) on a
  \(\Cst\)\nb-algebra~\(B\) with an \(S\)\nb-grading
  \((B_t)_{t\in S}\) \emph{preserves the grading} if
  \(E(B_t) \subseteq B_t\) for all \(t\in S\).
\end{definition}

\begin{lemma}
  \label{lem:centre_mult}
  Let \(A\to B\) be a non-degenerate embedding.  Extend it to a
  homomorphism \(\Mult(A) \to \Mult(B)\).  Let \(\tau\in\Mult(B)\).
  If~\(\tau\) commutes with~\(A\), then it commutes
  with~\(\Mult(A)\).
\end{lemma}

\begin{proof}
  Let \(m\in\Mult(A)\) and \(x\in A\).  Then
  \((\tau\cdot m)\cdot x = \tau (m x) = (m x) \tau = (m\cdot
  \tau)\cdot x\).  Since \(B = A\cdot B\), this implies
  \(\tau\cdot m = m\cdot \tau\) in \(\Mult(B)\).
\end{proof}

\begin{lemma}
  \label{lem:Exels_lemma}
  Let \(\tilde{A}\supseteq A\subseteq B\) be \(\Cst\)\nb-inclusions
  and let \(E\colon B \to \tilde{A}\) be a generalised expectation.
  Let \(M\in \Slice(A,B)\) be a slice and let
  \(\s(M) = M^* M \subseteq A\) be its source ideal.
  Then
  \(\tilde{A}_{M} \defeq \s(M)\cdot \tilde{A} \cdot
  \s(M)\) is a hereditary \(\Cst\)\nb-subalgebra
  in~\(\tilde{A}\).  There is a unique positive element
  \(\tau\in \Mult(\tilde{A}_{M})\) with \(\norm{\tau}\le 1\) and
  \(\tau\cdot m^* n = E(m^*)E(n) = m^* n\cdot \tau\) for all
  \(m,n\in M\).  And~\(\tau\) commutes with~\(A\) or, briefly,
  \(\tau\in \Mult(\tilde{A}_{M})\cap A'\).
\end{lemma}

\begin{proof}
  For a genuine conditional expectation \(B\to A\) on a separable
  \(\Cst\)\nb-algebra, this is
  \cite{Exel:noncomm.cartan}*{Lemma~11.5} for the slice~\(M^*\).
  We need a different proof.  The construction in
  \cite{Brown-Mingo-Shen:Quasi_multipliers}*{Remark~1.9} gives an
  approximate unit~\((\mu_\lambda)\) in~\(M^*M\) of the
  form
  \(\mu_\lambda = \sum_{i=1}^{n_\lambda} m_{\lambda,i}^*
  m_{\lambda,i}\) for some \(n_\lambda\in\N\) and
  \(m_{\lambda,1},\dotsc,m_{\lambda,n_\lambda}\in M\).  Let
  \[
    \tau_\lambda\defeq \sum_{i=1}^{n_\lambda} E(m_{\lambda,i})^*
    E(m_{\lambda,i}).
  \]
  The \(2\)\nb-positivity of~\(E\) implies
  \(E(m)^* E(m) \le E(m^* m)\) for all \(m\in B\) (see
  \cite{Choi:Schwarz}*{Corollary~2.8}).  Hence
  \[
    0 \le \tau_\lambda
    = \sum E(m_{\lambda,i}^*) E(m_{\lambda,i})
    \le \sum E(m_{\lambda,i}^* m_{\lambda,i})
    = \mu_\lambda\le 1.
  \]
  Let \(m,n\in M\).  Then
  \(m^* n, m_{\lambda,i} m^*\in A\).  As in the
  proof of \cite{Exel:noncomm.cartan}*{Lemma~11.5}, the
  \(A\)\nb-linearity of~\(E\) implies
  \begin{multline*}
    \tau_\lambda m^* n
    = \sum E(m_{\lambda,i}^*) E(m_{\lambda,i}) m^* n
      = \sum E(m_{\lambda,i}^*) E(m_{\lambda,i} m^* n)
    \\= \sum E(m_{\lambda,i}^* m_{\lambda,i} m^*)  E(n)
        = E(\mu_\lambda m^* ) E(n).
  \end{multline*}
  The adjoint of this relation is
  \(m^* n \tau_\lambda = E(m^*) E(n \mu_\lambda)\).  Thus
  \(\lim_{\lambda} \tau_\lambda \cdot m^* n \cdot x = E(m^*)E(n) \cdot
  x\) and
  \(\lim_{\lambda} x\cdot m^* n \cdot\tau_\lambda = x\cdot
  E(m^*)E(n)\) in the norm topology for all
  \(x\in \tilde{A}_{M}\).  Since the net~\((\tau_\lambda)\) is
  bounded and the embedding
  \(M^* M \hookrightarrow \tilde{A}_{M}\) is
  non-degenerate, it follows that~\((\tau_\lambda)\) converges
  strictly towards some \(\tau\in\Mult(\tilde{A}_{M})\), which
  satisfies \(\tau m^* n = E(m^*) E(n) = m^* n \tau\) for all
  \(m,n\in M\).  This determines the multiplier~\(\tau\) uniquely
  because \(M^* M\cdot \tilde{A}_{M}\cdot M^* M\)
  is dense in~\(\tilde{A}_{M}\).  The formula above shows
  that~\(\tau\) commutes with~\(M^* M\).  Then it commutes
  with all multipliers of~\(M^* M\) by
  Lemma~\ref{lem:centre_mult}.  Since the map
  \(A\to\Mult(\tilde{A}_{M})\) factors through
  \(\Mult(M^* M)\), it follows that~\(\tau\) commutes
  with~\(A\).
\end{proof}

\begin{lemma}
  \label{lem:conditional_expectation_preserving_slice}
  Let \(E\colon B \to A\) be a conditional expectation, and let
  \(M\in \Slice(A,B)\) be such that \(E(M)\subseteq M\).
  Then \(E(M)=M\cap A\), and this is a complemented ideal
  in~\(M^*M\).
\end{lemma}

\begin{proof}
  The inclusion \(E(M)\subseteq M\) implies
  \(E(M)\subseteq M\cap A\).  Since
  \(M\cap A\subseteq E(M)\) always holds, this implies
  \(E(M)=M\cap A\).  Clearly, \(M\cap A\) is an ideal
  in~\(M^*M\).  Let \(\tau\in \Mult(M^*M)\) be as in
  Lemma~\ref{lem:Exels_lemma} and let \(n,m\in M\).  Then
  \[
    \tau^2 m^* n
    = \tau E(m^*)E(n)
    = E^2(m^*)E^2(n)
    = E(m^*)E(n)
    =\tau m^* n.
  \]
  Thus~\(\tau\) is a projection, and
  \(\tau M^*M=E(M)^* E(M)=M\cap A\).
  Hence~\(M\cap A\) is complemented in~\(M^*M\).
\end{proof}


\begin{proposition}
  \label{pro:recognise_reduced_crossed_S_using_E}
  Let~\(B\) be a \(\Cst\)\nb-algebra with a saturated
  \(S\)\nb-grading \((B_t)_{t\in S}\) with unit fibre
  \(A\defeq B_1\) and a conditional expectation \(E\colon B\to A\).
  Let \(\pi\colon A\rtimes S\to B\) be the canonical epimorphism and
  let \(E_0\colon A\rtimes S \to A''\) be the canonical weak
  conditional expectation.  The following are equivalent:
  \begin{enumerate}
  \item\label{enu:exotic_characterisation1}%
    the \(S\)\nb-grading is wide and~\(E\) preserves it;
  \item \label{enu:exotic_characterisation1.5}%
    the \(S\)\nb-grading is wide,
    \(B_t=(B_t\cap A)\oplus B_t\cdot (B_t\cap A)^\bot\)
    and~\(E|_{B_t}\) is the projection onto the first summand for
    each \(t\in S\);
  \item \label{enu:exotic_characterisation2}%
    \(E\circ \pi = E_0\), and so~\((B_t)_{t\in S}\) is a closed
    action.
  \end{enumerate}
  Therefore, \(\pi\) descends to an isomorphism
  \(A\rtimes_\red S \cong B\) and \((B_t)_{t\in S}\) is a closed
  action if and only if the \(S\)\nb-grading is wide and~\(B\)
  admits an almost faithful conditional expectation~\(E\) that
  preserves the grading (and then \(E\) is in fact faithful).
\end{proposition}

\begin{proof}
  Assume first that \(E\circ \pi = E_0\).  Then
  \(E_0(A\rtimes S)\subseteq A\), that is, the \(S\)\nb-action
  on~\(A\) is closed.  And \(\ker \pi \subseteq \ker E_0\).  Since
  \(\ker (\Lambda\colon A\rtimes S \to A\rtimes_\red S)\) is the
  largest two-sided ideal contained in \(\ker E_0\), it follows that
  \(\ker \pi\subseteq \ker \Lambda\).  Thus there is a surjective
  homomorphism \(B \onto A\rtimes_\red S\) such that the composition
  \[
    A\rtimes_\alg S \hookrightarrow A\rtimes S
    \onto B
    \onto A\rtimes_\red S
  \]
  coincides with~\(\Lambda\), and \(B \onto A\rtimes_\red S\) is
  invertible if and only if~\(E\) is almost faithful (see
  \cite{Kwasniewski-Meyer:Essential}*{Lemma~3.10}).  If~\(E\) is
  almost faithful, it is faithful because the isomorphism
  \(B \cong A\rtimes_\red S\) identifies~\(E\) with~\(E_\red\); the
  latter is faithful by
  Proposition~\ref{prop:fast_intro_for_reduced}.

  The composite map \(A\rtimes_\alg S \to A\rtimes_\red S\) is
  injective by \cite{Buss-Exel-Meyer:Reduced}*{Proposition~4.3}.
  Hence so is the map \(A\rtimes_\alg S \to B\).  If \(t,u\in S\),
  then the intersection of \(B_t\) and~\(B_u\)
  in~\(A\rtimes_\alg S\) is the closed linear span of~\(B_v\) for
  \(v\le t,u\).  Therefore, the \(S\)\nb-grading on~\(B\) is wide.
  So the ideal \(I_{1,t}=\overline{\sum_{v\le 1,t} B_v}\) coincides
  with \(B_t \cap A\).  Since the action~\((B_t)_{t\in S}\) is
  closed, the last part of Proposition
  \ref{prop:conditional_expectation} implies that
  \(B_t=(B_t\cap A)\oplus B_t\cdot (B_t\cap A)^\bot\)
  and~\(E_0|_{B_t}\) is the projection onto the first summand.
  Since \(E\circ \pi = E_0\) and~\(\pi\) is the identity map
  on~\(B_t\) the same holds for~\(E|_{B_t}\).  This shows the last
  part of the assertion and that~\ref{enu:exotic_characterisation2}
  implies~\ref{enu:exotic_characterisation1.5}.

  It is trivial that \ref{enu:exotic_characterisation1.5}
  implies~\ref{enu:exotic_characterisation1}.  We show
  that~\ref{enu:exotic_characterisation1}
  implies~\ref{enu:exotic_characterisation2}.  Thus assume that the
  \(S\)\nb-grading is wide and that the conditional expectation
  \(E\colon B\to A\) preserves it.  Let \(t\in S\).  Then the
  ideal~\(I_{1,t}\) is equal to \(B_t\cap A\).  Since
  \(B_t\in \Slice(A,B)\) is a slice,
  Lemma~\ref{lem:conditional_expectation_preserving_slice} implies
  \(E(B_t)=B_t\cap A\) and that the ideal \(I_{1,t}=E(B_t)\) is
  complemented in \(\s(B_t)=B_t^*B_t\).  Therefore, the
  action~\((B_t)_{t\in S}\) is closed by Proposition
  \ref{prop:conditional_expectation}.  The map~\(E\) is the identity
  on \(B_t\cdot I_{1,t}=B_t\cap A=I_{1,t}=E(B_t)\).  It vanishes on
  \(B_t\cdot I_{1,t}^\bot\) because
  \[
    E(B_tI_{1,t}^\bot)
    = E(B_tI_{1,t}^\bot) (B_t\cap A)
    = E(B_tI_{1,t}^\bot (B_t\cap A))
    = E(B_tI_{1,t}^\bot I_{1,t} )=0.
  \]
  So \(E\circ \pi = E_0\) by the last part of Proposition \ref{prop:conditional_expectation}.
\end{proof}

Proposition~\ref{pro:recognise_reduced_crossed_S_using_E} generalises
\cite{Exel:Amenability}*{Theorem~3.3} in the case of group gradings.

\begin{corollary}
  \label{cor:exotic_characterisation}
  Let \(A\subseteq B\) be a \(\Cst\)\nb-inclusion.  There are an
  inverse semigroup~\(S\), a closed \(S\)\nb-action on~\(A\), and an
  isomorphism \(B \cong A\rtimes_\red S\) if and only if~\(A\) is
  the unit fibre of an inverse semigroup grading on~\(B\) that is
  preserved by an almost faithful conditional expectation
  \(E\colon B\to A\).
\end{corollary}

\begin{proof}
  Compared to
  Proposition~\ref{pro:recognise_reduced_crossed_S_using_E}, only
  the assumption that the grading is wide and saturated is missing.
  We remedy this by refining the grading.  Let
  \(\tilde{S}\defeq \setgiven{B_t J}{t\in S,\ J\in \Ideals(A)}\).
  This is an inverse subsemigroup of \(\Slice(A,B)\).  So~\(B\) is
  \(\tilde{S}\)\nb-graded in a tautological way.  If
  \(t,u\in \tilde{S}\), then \(B_t \cap B_u = B_t \cdot J\) for some
  ideal \(J\in\Ideals(A)\) because any Hilbert subbimodule
  of~\(B_t\) has this form.  So \(B_t \cap B_u \in \tilde{S}\) and
  the \(\tilde{S}\)\nb-grading on~\(B\) is wide.  It is saturated
  because \(\tilde{S}\subseteq \Slice(A,B)\).  It is still preserved
  by~\(E\) because
  \(E(B_t\cdot J)=E(B_t)\cdot J\subseteq B_t\cdot J \subseteq B_t\).
  Now Proposition~\ref{pro:recognise_reduced_crossed_S_using_E}
  finishes the proof.
\end{proof}

\begin{corollary}
  \label{cor:unique_conditional_grading_preserving}
  Let~\(A\) be the unit fibre of an inverse semigroup grading
  \((B_t)_{t\in S}\) on~\(B\).  At most one conditional expectation
  \(E\colon B\to A\) preserves this grading.
\end{corollary}

\begin{proof}
  Let \(\tilde{S}\defeq \setgiven{B_t J}{t\in S,\ J\in \Ideals(A)}\)
  be the wide saturated grading of~\(B\) as in the proof of
  Corollary~\ref{cor:exotic_characterisation}.  Let
  \(\pi\colon A\rtimes \tilde{S}\to B\) be the canonical epimorphism
  and let \(E_0\colon A\rtimes \tilde{S} \to A\) be the canonical
  conditional expectation.  Then \(E\circ \pi = E_0\) by
  Proposition~\ref{pro:recognise_reduced_crossed_S_using_E}, and
  this determines~\(E\) because~\(\pi\) is surjective.
\end{proof}

\begin{example}
  An \(S\)\nb-action need not be closed if there is a conditional
  expectation \(P\colon A\rtimes S\to A\) that does not preserve the canonical \(S\)\nb-grading
  of~\(A\rtimes S\).
  By~\cite{BussExel:Fell.Bundle.and.Twisted.Groupoids}*{Proposition~5.3},
  there is a non-Hausdorff étale groupoid~\(H\) with a compact
  Haudorff object space~\(X\) and a faithful conditional expectation
  from \(B\defeq\Cst_\red(H)\) onto \(A\defeq \Cont(X)\).  Let~\(S\) be
  the inverse semigroup of bisections of~\(H\) and let~\(S\) act on
  \(\Cont(X)\) in the canonical way.  Then
  \(\Cst_\red(H) \cong \Cont(X) \rtimes_\red S\).  This \(S\)\nb-action on
  \(\Cont(X)\) is not closed, although there is a faithful conditional
  expectation \(E\colon \Cont(X)\rtimes_\red S\to \Cont(X)\).
\end{example}

\section{Exel's noncommutative Cartan subalgebras}
\label{sec:Exel}

Noncommutative Cartan subalgebras were introduced
in~\cite{Exel:noncomm.cartan}, as a generalisation of the
(commutative) Cartan subalgebras studied by
Renault~\cite{Renault:Cartan.Subalgebras}.  Note that Exel considers
a subalgebra \(B\subseteq A\), so the roles of \(A\) and~\(B\) are
exchanged in~\cite{Exel:noncomm.cartan}.  And Exel assumes
throughout that the larger \(\Cst\)\nb-algebra is separable.

\begin{definition}[\cite{Exel:noncomm.cartan}*{Definition~9.2} and
  \cite{Exel:noncomm.cartan}*{Definition~12.1}]
  \label{def:virtual_commutant}
  A \emph{virtual commutant} of~\(A\) in~\(B\) is an
  \(A\)\nb-bimodule map \(J\to B\) defined on an ideal
  \(J\in\Ideals(A)\).  An inclusion \(A\subseteq B\) is a
  \emph{noncommutative Cartan subalgebra} if it is regular, any
  virtual commutant has range in~\(A\), and there is an almost
  faithful conditional expectation \(E\colon B\to A\).
\end{definition}
\begin{definition}[\cite{Kwasniewski-Meyer:Aperiodicity}]
  A Hilbert \(A\)\nb-bimodule~\(\Hilm[H]\) over a
  \(\Cst\)\nb-algebra~\(A\) is \emph{purely outer} if there is no
  non-zero ideal \(J\in\Ideals(A)\) with \(\Hilm[H]\cdot J \cong J\)
  as a Hilbert bimodule.  An action~\(\Hilm\) of an inverse
  semigroup on a \(\Cst\)\nb-algebra~\(A\) is \emph{purely outer} if
  the Hilbert \(A\)\nb-bimodules \(\Hilm_t\cdot I_{1,t}^\bot\) are
  purely outer for all \(t\in S\), where
  \(I_{1,t}^\bot = \setgiven{a\in A}{a\cdot I_{1,t}=0}\).
\end{definition}

\begin{theorem}
  \label{the:nc_Cartan}
  Let \(A\subseteq B\) be a regular \(\Cst\)\nb-subalgebra with an
  almost faithful conditional expectation \(E\colon B\to A\).  The
  following conditions are equivalent:
  \begin{enumerate}
  \item \label{the:nc_Cartan0}%
    \(E\) is faithful and there is exactly one conditional
    expectation \(I B I \to I\) for all \(I\in\Ideals(A)\), namely,
    the restriction of~\(E\);
  \item \label{the:nc_Cartan1}%
    there is exactly one faithful conditional expectation
    \(I B I \to I\) for all \(I\in\Ideals(A)\);
  \item \label{the:nc_Cartan2}%
    any virtual commutant in~\(B\) has range in~\(A\), that is,
    \(A\subseteq B\) is a noncommutative Cartan subalgebra;
  \item \label{the:nc_Cartan3}%
    \(I' \cap \Mult(I B I) =Z\Mult(I)\) for all \(I\in\Ideals(A)\);
  \item \label{the:nc_Cartan4}%
    if a slice \(M\subseteq B\) is isomorphic to an ideal~\(I\)
    in~\(A\) as a Hilbert \(A\)\nb-bimodule, then already \(M=I\) as
    a subset of~\(B\);
  \item \label{the:nc_Cartan5}%
    if a slice \(M\subseteq B\) satisfies \(M\otimes_A M \cong M\)
    as a Hilbert \(A\)\nb-bimodule, then \(M\cdot M = M\);
  \item \label{the:nc_Cartan6}%
    for any unital inverse semigroup~\(S\) and any saturated, wide
    grading \((B_t)_{t\in S}\) on~\(B\) with unit fibre~\(A\), the
    action~\((B_t)_{t\in S}\) of~\(S\) on~\(A\) is closed and purely
    outer, and the canonical \Star{}homomorphism
    \(\pi\colon A\rtimes S\to B\) descends to an isomorphism
    \(A\rtimes_\red S \congto B\);
  \item \label{the:nc_Cartan7}%
    there are a unital inverse semigroup~\(S\), a closed and purely
    outer action of~\(S\) on~\(A\), and an isomorphism
    \(A\rtimes_\red S \congto B\) mapping~\(A\) identically to
    itself.
  \end{enumerate}
  In particular, if the above equivalent conditions hold, then
  \(E\colon B\to A\) faithful and it is the only conditional expectation onto~\(A\).
\end{theorem}

Condition~\ref{the:nc_Cartan7} without ``purely outer'' is the
conclusion of the main theorem in~\cite{Exel:noncomm.cartan}.  So
the equivalence of \ref{the:nc_Cartan2} and~\ref{the:nc_Cartan7}
contains the main result of~\cite{Exel:noncomm.cartan} and, together
with~\ref{the:nc_Cartan6}, characterises precisely to which inverse
semigroup actions Exel's theory applies.  Conditions
\ref{the:nc_Cartan4} and~\ref{the:nc_Cartan5} are slightly different
ways of saying that the canonical action of \(\Slice(A,B)\) on~\(A\)
is purely outer.  This gives a candidate for the inverse semigroup
action in~\ref{the:nc_Cartan7}.  Conditions \ref{the:nc_Cartan0}
and~\ref{the:nc_Cartan1} concern the uniqueness of conditional
expectations studied by Zarikian~\cite{Zarikian:Unique_expectations}
--~but for all inclusions \(I \subseteq I B I\) for
\(I\in\Ideals(A)\) and not just for the inclusion \(A\subseteq B\).
Condition~\ref{the:nc_Cartan3} is equivalent to
\(I' \cap \Mult(I B I) \subseteq \Mult(I)\) by
Lemma~\ref{lem:centre_mult}.  If~\(A\) is simple, this says that any
multiplier of~\(B\) that commutes with~\(A\) is already a scalar
multiple of~\(1\).  If~\(A\) is commutative, then it is equivalent
to~\(A\) being maximal Abelian in~\(B\) (compare
Corollary~\ref{cor:Cartan}).

The proof of the theorem requires some preparation.  Our first goal
is Proposition~\ref{pro:virtual_commutant_slices}, which says that
\ref{the:nc_Cartan2}--\ref{the:nc_Cartan5} in
Theorem~\ref{the:nc_Cartan} are equivalent for any non-degenerate
\(\Cst\)\nb-inclusion, without assuming regularity or a conditional
expectation.

\begin{lemma}
  \label{lem:virtual_commutant_multiplier}
  Let \(A\subseteq B\) be a non-degenerate inclusion and
  \(I\in\Ideals(A)\).  If \(\tau\in \Mult(I B I)\cap I'\), then the
  map \(I\to B\), \(x\mapsto \tau\cdot x\), is a virtual commutant.
  Conversely, any virtual commutant \(\varphi\colon I\to B\) is of
  this form for a unique \(\tau\in \Mult(I B I)\cap I'\).  And
  \(\varphi(I) \subseteq I\) if and only if \(\tau\in Z\Mult(I)\).
\end{lemma}

\begin{proof}
  If \(\tau\in \Mult(I B I)\cap I'\), then \(I\to B\),
  \(x\mapsto \tau\cdot x = x\cdot \tau\), is a virtual commutant by
  direct computation.  Conversely, let \(\varphi\colon I\to B\) be a
  virtual commutant.  Then \(\varphi(I) \subseteq I B I\)
  because~\(\varphi\) is \(I\)\nb-bilinear.  Represent~\(B\)
  faithfully in \(\Bound(\Hils)\) for a Hilbert space~\(\Hils\).
  There is \(\tau\in \Bound(\Hils)\) with
  \(\varphi(x) = \tau\cdot x = x\cdot \tau\) for all \(x\in I\) by
  \cite{Exel:noncomm.cartan}*{Theorem~9.5}.  And
  \(\tau\cdot I \subseteq I B I\) by
  \cite{Exel:noncomm.cartan}*{Proposition~9.3}.  Hence
  \(\tau\cdot I B I = \varphi(I) B I \subseteq I B I B I \subseteq I
  B I\) and
  \(I B I\cdot \tau = I B \varphi(I) \subseteq I B I B I \subseteq I
  B I\).  So \(\tau\in\Mult(I B I)\), and~\(\varphi\) has the
  asserted form.  Let \(\tau_1,\tau_2\in \Mult(I B I)\) satisfy
  \(\tau_1\cdot x = \tau_2\cdot x\) for all \(x\in I\).  Then
  \(\tau_1\cdot x y = \tau_2\cdot x y\) for all \(x\in I\),
  \(y\in B I\), so that \(\tau_1 =\tau_2\) as multipliers
  of~\(I B I\).  Thus \(\tau\in\Mult(I B I)\) above is unique.
  Clearly, \(\tau\in Z\Mult(I)\) if and only if
  \(\varphi(I) \subseteq I\).  This is equivalent to
  \(\varphi(I) \subseteq A\) because~\(\varphi\) is
  \(A\)\nb-bilinear and \(I^2 = I\).
\end{proof}

\begin{proposition}
  \label{pro:virtual_commutant_slices}
  Let \(A\subseteq B\) be a non-degenerate inclusion and
  \(I\in\Ideals(A)\).  The following are equivalent:
  \begin{enumerate}
  \item \label{pro:virtual_commutant_slices_1}%
    any virtual commutant \(\varphi\colon I\to B\) has range
    in~\(A\);
  \item \label{pro:virtual_commutant_slices_2}%
    \(\Mult(I B I) \cap I' = Z\Mult(I)\);
  \item \label{pro:virtual_commutant_slices_3}%
    if a slice \(M\subseteq B\) is isomorphic to~\(I\) as a Hilbert
    \(A\)\nb-bimodule, then \(M=I\);
  \item \label{pro:virtual_commutant_slices_4}%
    if a slice \(M\subseteq B\) satisfies \(M^*M=I\) and
    \(M\otimes_A M \cong M\) as a Hilbert \(A\)\nb-bimodule, then
    already \(M\cdot M = M\).
  \end{enumerate}
\end{proposition}

\begin{proof}
  Lemma~\ref{lem:virtual_commutant_multiplier} shows that
  \ref{pro:virtual_commutant_slices_1}
  and~\ref{pro:virtual_commutant_slices_2} are equivalent.  If
  \(M\subseteq B\) is a slice as
  in~\ref{pro:virtual_commutant_slices_3}, then the isomorphism
  \(I\cong M\) is a virtual commutant.
  So~\ref{pro:virtual_commutant_slices_1}
  implies~\ref{pro:virtual_commutant_slices_3}.  We are going to
  show that not~\ref{pro:virtual_commutant_slices_2} implies
  not~\ref{pro:virtual_commutant_slices_3}.  This will complete the
  proof that
  \ref{pro:virtual_commutant_slices_1}--\ref{pro:virtual_commutant_slices_3}
  are equivalent.  A unital \(\Cst\)\nb-algebra is spanned by its
  unitary elements.  Therefore,
  \(\Mult(I B I)\cap I' \neq Z\Mult(I)\) if and only if there is a
  unitary \(\tau\in \Mult(I B I)\cap I'\) with
  \(\tau \notin Z\Mult(I)\).  Then \((\tau x)^* (\tau y) = x^* y\)
  and \((\tau x) (\tau y)^* = (x \tau) (y \tau)^* = x y^*\) for all
  \(x,y \in I\), and \(a\cdot (\tau x) = a x \tau = \tau (a x)\) and
  \((\tau x)\cdot a = \tau (x a)\) for all \(x\in I\), \(a\in A\).
  So \(\tau\cdot I \subseteq B\) is a slice that is isomorphic
  to~\(I\) as a Hilbert \(A\)\nb-bimodule, but not contained
  in~\(A\).

  We prove \ref{pro:virtual_commutant_slices_3}\(\iff\)%
  \ref{pro:virtual_commutant_slices_4}.  A slice~\(M\) contained
  in~\(A\) is a closed \(A\)\nb-bimodule, that is, a closed
  two-sided ideal in~\(A\).  The slices form an inverse semigroup
  (compare Proposition~\ref{prop:regular_vs_inverse_semigroups}).
  If a slice~\(M\) satisfies \(M\cdot M = M\), then
  \(M\cdot M\cdot M = M\) implies \(M = M^*\).  Hence
  \(M = M\cdot M = M\cdot M^* \subseteq A\).  So a slice~\(M\) is
  contained in~\(A\) if and only if \(M\cdot M = M\).  A Hilbert
  \(A\)\nb-bimodule~\(M\) satisfies \(M\otimes_A M \cong M\) if and
  only if \(M\cong I\) for some ideal \(I\in\Ideals(A)\) (see
  \cite{Buss-Meyer:Actions_groupoids}*{Proposition~4.6}).
  Therefore, the two conditions whose equivalence is required
  in~\ref{pro:virtual_commutant_slices_3} are equivalent to the two
  conditions whose equivalence is required
  in~\ref{pro:virtual_commutant_slices_4}.
\end{proof}

Now we turn to preparatory results about conditional expectations.

\begin{lemma}
  \label{lem:multipliers_conditions_expectations}
  Let \(\tilde{A}\supseteq A\subseteq B\) be non-degenerate
  \(\Cst\)\nb-inclusions and let \(E\colon B \to \tilde{A}\) be a
  generalised expectation.  Then~\(E\) is non-degenerate, and it
  extends uniquely to a strictly continuous generalised expectation
  \(\bar{E} \colon \Mult(B) \to \Mult(\tilde{A})\) for the
  induced inclusion \(\Mult(A) \subseteq \Mult(B)\).  If one of the
  maps \(E\) and~\(\bar{E}\) is faithful, symmetric, or almost
  faithful, then so is the other.
\end{lemma}

\begin{proof}
  The non-degeneracy of the inclusions of~\(A\) means that an
  approximate unit for~\(A\) is also one for \(\tilde{A}\)
  and~\(B\).  Since~\(E|_A\) is the identity, it maps an approximate
  unit for~\(B\) to one for~\(\tilde{A}\).  This is the notion of
  non-degeneracy used in
  \cite{Lance:Hilbert_modules}*{Corollary~5.7} to extend~\(E\) to
  \(\Mult(B)\).  More precisely, it is shown
  in~\cite{Lance:Hilbert_modules} that the extension~\(\bar{E}\) is
  strictly continuous on the unit ball.  Using the Cohen--Hewitt
  Factorisation Theorem in this argument allows to prove strict
  continuity everywhere.

  Suppose that~\(\bar{E}\) is almost faithful.  Let \(a\in B\) be
  such that \(E((ab)^*ab)=0\) for all \(b\in B\) and let
  \(m\in \Mult(B)\).  Let~\((\mu_\lambda)\) be an approximate unit
  for~\(A\).  We compute
  \[
    E((am)^*am)
    = \lim \mu_\lambda E((am)^*am)\mu_\lambda
    = \lim  E((a(m\mu_\lambda))^*a(m\mu_\lambda)) = 0.
  \]
  This implies \(a=0\).  Therefore, \(E\) is almost faithful.
  Conversely, assume~\(E\) to be almost faithful.  Let
  \(m\in \Mult(B)\) satisfy \(\bar{E}((mn)^*mn)=0\) for all
  \(n\in \Mult(B)\).  If \(a,b\in A\), then
  \(E\bigl( ((ma)b)^*((ma)b)\bigr) = \bar{E}\bigl((m
  (ab))^*(m(ab))\bigr)=0\).  Thus \(m\cdot a=0\) because~\(E\) is
  almost faithful.  This implies \(m=0\).  Hence~\(E\) is almost
  faithful.

  Clearly, \(E\) is symmetric if~\(\bar{E}\) is.  Conversely,
  assume~\(E\) to be symmetric.  Let~\((\mu_\lambda)\) be an
  approximate unit for~\(B\).  Then~\((\mu_\lambda m)\) for
  \(m\in \Mult(B)\) converges strictly to~\(m\).  Assume
  \(\bar{E}(m^*m)=0\).  Then
  \(E((\mu_\lambda m)^* \mu_\lambda m)=0\) because
  \(E(m^*\mu_\lambda^2 m)\le \bar{E}(m^*m)=0\).  Hence
  \(E(\mu_\lambda m m^*\mu_\lambda)=0\) because~\(E\) is symmetric.
  Since~\(\bar{E}\) is strictly continuous, this implies
  \(\bar{E}(m m^*) = \stlim E( \mu_\lambda (mm^*) \mu_\lambda)=0\).
  Thus~\(\bar{E}\) is symmetric.  A conditional expectation is
  faithful if and only if it is almost faithful and symmetric (see
  \cite{Kwasniewski-Meyer:Essential}*{Corollary~3.7}).  Hence~\(E\)
  is faithful if and only if~\(\bar{E}\) is.
\end{proof}

\begin{lemma}[compare \cite{Zarikian:Unique_expectations}*{Proposition~3.1}]
  \label{lem:multiplicative_domain}
  Let \(\tilde{A}\supseteq A\subseteq B\) be non-degenerate
  \(\Cst\)\nb-inclusions and let \(E\colon B \to \tilde{A}\) be a
  generalised expectation.  Assume that~\(E\) is a unique
  generalised  expectation \(B\to\tilde{A}\), or
  that~\(E\) is a unique faithful or a unique almost faithful generalised
  expectation \(B\to\tilde{A}\).  Let
  \(\bar{E}\colon \Mult(B)\to \Mult(\tilde{A})\) be the unique
  extension of~\(E\).  Then \(A'\cap \Mult(B)\) is contained in the
  multiplicative domain of~\(\bar{E}\) and~\(\bar{E}\) restricts to
  a \Star{}homomorphism
  \(A'\cap \Mult(B) \to A' \cap \Mult(\tilde{A})\).
\end{lemma}

\begin{proof}
  Let \(x\in A'\cap \Mult(B)\) satisfy \(x^*=x\) and \(\norm{x}<1\).
  Then \(\norm*{\bar{E}(x)}<1\).  So we may define a completely
  positive map \(E_x\colon \Mult(B)\to \Mult(\tilde{A})\) by
  \[
    E_x(t)\defeq (1-\bar{E}(x))^{-1/2} \cdot
    \bar{E}((1-x)^{1/2} t (1-x)^{1/2})\cdot (1-\bar{E}(x))^{-1/2}.
  \]
  Let \(a\in A\).  Then \(a x = x a\) and hence
  \(a \bar{E}(x) = \bar{E}(a x) = \bar{E}(x a) = \bar{E}(x) a\), and
  the same holds for \((1-x)^{\pm1/2}\) instead of~\(x\).  This
  implies \(E_x(a) = a\).  Then it follows that~\(E_x\) is
  contractive and a generalised expectation.  It maps~\(B\)
  to~\(\tilde{A}\) and is strictly continuous as well.  So~\(E_x\)
  is the unique strictly continuous extension of its restriction
  \(E_x|_B\colon B\to \tilde{A}\) (see
  Lemma~\ref{lem:multipliers_conditions_expectations}).  Since
  \((1-x)^{1/2}\) is invertible, \(E_x\) is faithful if~\(\bar{E}\)
  is, and~\(E_x\) is almost faithful if~\(\bar{E}\) is.  Thus
  \(E=E_x|_B\) by the uniqueness assumption in the lemma.  This
  implies \(\bar{E} = E_x\) and then \(\bar{E}(x) = E_x(x)\).  So
  \begin{multline*}
    \bar{E}(x) - \bar{E}(x)^2
    = (1-\bar{E}(x))^{1/2} \bar{E}(x)(1-\bar{E}(x))^{1/2}
    \\= (1-\bar{E}(x))^{1/2} E_x(x)(1-\bar{E}(x))^{1/2}
    = \bar{E}((1-x)^{1/2} x (1-x)^{1/2})
    = \bar{E}(x-x^2).
  \end{multline*}
  This is equivalent to \(\bar{E}(x^*)\bar{E}(x)=\bar{E}(x^* x)\).
  Hence~\(x\) is in the multiplicative domain of~\(\bar{E}\) (see
  \cite{Choi:Schwarz}*{Theorem~3.1}).  It follows that the
  multiplicative domain contains all self-adjoint elements of
  \(A'\cap \Mult(B)\).  Thus~\(\bar{E}\) is multiplicative on
  \(A'\cap \Mult(B)\).  Being \(A\)\nb-bilinear, it must map
  \(A'\cap \Mult(B)\) into \(A' \cap \Mult(\tilde{A})\).
\end{proof}

\begin{proposition}
  \label{prop:uniqueness_implies_masa}
  Let \(A\subseteq B\) be non-degenerate.  Assume that there is
  exactly one faithful conditional expectation
  \(E\colon B\to A\subseteq B\).  Then
  \(A'\cap \Mult(B)=Z(\Mult(A))\).
\end{proposition}

\begin{proof}
  Extend~\(E\) to multipliers as in
  Lemma~\ref{lem:multipliers_conditions_expectations}.
  Lemma~\ref{lem:centre_mult} implies
  \(Z(\Mult(A)) = A' \cap \Mult(A) \subseteq A'\cap \Mult(B)\).
  And~\(\bar{E}\) restricts to a \Star{}homomorphism
  \(e\colon A'\cap \Mult(B) \to Z(\Mult(A))\) by
  Lemma~\ref{lem:multiplicative_domain}.  We claim that~\(e\) is
  injective when~\(\bar{E}\) is faithful.  Otherwise, there
  is \(a\in A'\cap \Mult(B)\) with \(a\ge 0\), \(a \neq 0\) and
  \(e(a) = \bar{E}(a)=0\).  Then \(\bar{E}(a^* a)=0\) for all
  because~\(a\) is in the multiplicative domain
  of~\(\bar{E}\).  Since~\(\bar{E}\) is faithful, this
  implies \(a=0\).  So~\(e\) is injective.
  Since \(e^2 = e\), it follows that~\(e\) is invertible.  Hence
  \(Z(\Mult(A))=A'\cap \Mult(B)\).
\end{proof}

\begin{lemma}
  \label{lem:expect_to_virtual_commutant}
  Let \(P\colon B\to A\) be a conditional expectation and let
  \(M\subseteq B\) be a slice.  Let \(\tau\in\Mult(M^* M)\) be as in
  Lemma~\textup{\ref{lem:Exels_lemma}}.  Then there is a unique
  isometric virtual commutant~\(\varphi\) that maps the closure
  of~\(P(M)\) into~\(M\) and that satisfies
  \(\varphi(P(m)) = m\cdot \tau^{1/2}\) for all \(m\in M\).
\end{lemma}

\begin{proof}
  Up to a change in conventions, this is
  \cite{Exel:noncomm.cartan}*{Proposition~11.14}.  The proof
  in~\cite{Exel:noncomm.cartan} is literally the same.  The closure
  of~\(P(M)\) is a closed ideal.  Lemma~\ref{lem:Exels_lemma} gives
  \(\tau \in \Mult(M^* M)\) with
  \(P(m)^* P(n) = \tau^{1/2} m^* n \tau^{1/2}\) for all
  \(m,n\in M\).  So the map \(P(m) \to m\cdot \tau^{1/2}\) is
  isometric.  Then it extends to the closure.  The extension is
  \(A\)\nb-bilinear because~\(\tau\) commutes with~\(M^* M\) and
  hence with all multipliers of~\(M^* M\) by
  Lemma~\ref{lem:centre_mult}.  The values of~\(\varphi\) belong
  to~\(M\) because \(M\cdot \Mult(M^* M) \subseteq M\).
\end{proof}

\begin{lemma}
  \label{lem:expectation_unique_from_no_commutants}
  Let \(P\colon B\to A\) be a conditional expectation.  If any
  virtual commutant in~\(B\) has range in~\(A\), then~\(P\)
  preserves slices, that is, \(P(M) \subseteq M\) for any slice
  \(M\subseteq B\).
\end{lemma}

\begin{proof}
  The proof follows~\cite{Exel:noncomm.cartan}.
  Lemma~\ref{lem:Exels_lemma} provides a central, positive
  multiplier \(\tau\in\Mult(M^* M)\) with \(\norm{\tau}\le 1\) and
  \(\tau m^* n = P(m^*)P(n) = m^* n \tau\) for all \(n,m\in M\).
  Let~\(I\) be the closure of~\(P(M)\).
  Lemma~\ref{lem:expect_to_virtual_commutant} gives an isometric
  virtual commutant \(\varphi\colon I \to B\) with
  \(\varphi(P(m)) \defeq m\cdot \tau^{1/2}\) for all \(m\in M\).  By
  assumption, \(M\cdot \tau^{1/2} \subseteq I\).  The ideal~\(I\) is
  contained in \(M^* M = \s(M)\) because an approximate unit
  for~\(\s(M)\) is also one for~\(I\).  Hence~\(\tau\) restricts to
  a multiplier of~\(I\).  So \(I\cdot \tau^{1/2} \subseteq I\).
  Then
  \(M\cdot \tau \subseteq I\cdot \tau^{1/2} \subseteq I \subseteq
  A\).  So \(\tau M^* M = (M \tau)^* M \subseteq A M = M\).  Then
  \(P(m)^* P(n) = P(m^*) P(n) = \tau m^* n\in M\) for all
  \(m,n\in M\).  Since~\(I\) is the closure of \(P(M)\) and a closed
  two-sided ideal, this implies \(I = I^* I \subseteq M\).
  That is, \(P(M) \subseteq M\).
\end{proof}

\begin{lemma}
  \label{lem:Hilbert_module_outerness}
  Let~\(\Hilm[H]\) be a Hilbert \(A\)\nb-bimodule over a
  \(\Cst\)\nb-algebra~\(A\) and let \(J\in\Ideals(A)\) be an ideal.
  The following conditions are equivalent and imply that~\(J\) is
  \(\Hilm[H]\)\nb-invariant:
  \begin{enumerate}
  \item \label{lem:Hilbert_module_outerness_1}%
    some subbimodule of~\(J\Hilm[H]J\) is isomorphic to the trivial
    Hilbert \(A\)\nb-bimodule~\(J\);
  \item \label{lem:Hilbert_module_outerness_2}%
    \(J\Hilm[H]J\cong J\) as Hilbert \(A\)-bimodules;
  \item \label{lem:Hilbert_module_outerness_3}%
    \(\Hilm[H]J\cong J\) as Hilbert \(A\)-bimodules.
  \end{enumerate}
\end{lemma}

\begin{proof}
  By the Rieffel correspondence, closed subbimodules
  in~\(J\Hilm[H]J\) are naturally in bijection with ideals in
  \(\rg(\Hilm[H]) \cap J\) and with ideals in
  \(\s(\Hilm[H]) \cap J\).  So the only subbimodule that can be
  isomorphic to~\(J\) is~\(J\Hilm[H]J\) itself, and this requires
  \(J\subseteq \s(\Hilm[H]) \cap \rg(\Hilm[H])\).  Thus
  \ref{lem:Hilbert_module_outerness_1}
  and~\ref{lem:Hilbert_module_outerness_2} are equivalent.  If
  \(\Hilm[H]J\cong J\), then \(J\Hilm[H]J = J^2 = J\) as well.
  Hence \ref{lem:Hilbert_module_outerness_3} implies
  \ref{lem:Hilbert_module_outerness_2}.  Conversely, if
  \(J\Hilm[H]J\cong J\), then~\(J\) is a closed subbimodule of the
  Hilbert bimodule~\(\Hilm[H]J\).  There is no room for it to be a
  proper subbimodule because \(\s(\Hilm[H]J) \subseteq J\).  So
  \(\Hilm[H]J= J\Hilm[H]J\cong J\).  Similarly,
  \(J\Hilm[H]J\cong J\) implies \(J\Hilm[H]= J\Hilm[H]J\), which is
  the same as~\(\Hilm[H]J\).  Thus~\(J\) is
  \(\Hilm[H]\)\nb-invariant.
\end{proof}

\begin{proof}[Proof of Theorem~\textup{\ref{the:nc_Cartan}}]
  It is clear that \ref{the:nc_Cartan0}
  implies~\ref{the:nc_Cartan1}.
  Proposition~\ref{prop:uniqueness_implies_masa} applied to the
  non-degenerate inclusions \(I\to I B I\) for \(I\in\Ideals(A)\)
  shows that~\ref{the:nc_Cartan1} implies~\ref{the:nc_Cartan3}.
  Conditions \ref{the:nc_Cartan2}--\ref{the:nc_Cartan5} are
  equivalent by Proposition~\ref{pro:virtual_commutant_slices}.

  Now assume the equivalent conditions
  \ref{the:nc_Cartan2}--\ref{the:nc_Cartan5}.  We will show that
  they imply~\ref{the:nc_Cartan6} and~\ref{the:nc_Cartan0}.  Let
  \((B_t)_{t\in S}\) be any wide, saturated \(S\)\nb-grading
  on~\(B\).  If a Hilbert subbimodule of \(B_t \cdot I_{1,t}^\bot\)
  is isomorphic as an \(A\)\nb-bimodule to an ideal in~\(A\), then
  it is already contained in~\(A\) by~\ref{the:nc_Cartan4}.  Then it
  is contained in \(B_t \cap A = I_{1,t}\), which is orthogonal to
  \(B_t \cdot I_{1,t}^\bot\).  Therefore, the action on~\(A\)
  defined by the \(S\)\nb-grading on~\(B\) is purely outer.
  Lemma~\ref{lem:expectation_unique_from_no_commutants} already
  shows that the given conditional expectation \(E\colon B\to A\)
  preserves all slices and hence also the grading
  \((B_t)_{t\in S}\).  Then
  Proposition~\ref{pro:recognise_reduced_crossed_S_using_E} implies
  that the canonical surjection \(A\rtimes S \onto B\) descends to a
  \Star{}isomorphism \(A\rtimes_\red S \congto B\).  Then~\(E\) is
  faithful.  Therefore, \ref{the:nc_Cartan2}--\ref{the:nc_Cartan5}
  imply~\ref{the:nc_Cartan6}, and to show that they
  imply~\ref{the:nc_Cartan0} it suffices to prove that the
  conditional expectations \(I B I \to I\) are unique for all
  \(I\in\Ideals(A)\).  To this end, take \(I\in\Ideals(A)\) and let
  \(P\colon I B I\to I\) be any conditional expectation.  The
  inclusion \(I \subseteq I B I\) is regular as well,
  and~\ref{the:nc_Cartan2} and
  Lemma~\ref{lem:expectation_unique_from_no_commutants} imply
  that~\(P\) preserves slices.  Thus there is at most one
  conditional expectation \(I B I \to I\) by
  Corollary~\ref{cor:unique_conditional_grading_preserving}.

  Hence we shown that the conditions
  \ref{the:nc_Cartan0}--\ref{the:nc_Cartan5} are equivalent and they
  imply~\ref{the:nc_Cartan6}.  It is trivial
  that~\ref{the:nc_Cartan6} implies~\ref{the:nc_Cartan7}.  The proof
  of the theorem will be finished by showing
  that~\ref{the:nc_Cartan7} implies~\ref{the:nc_Cartan2}.

  Let~\(S\) be a unital inverse semigroup and let~\(\Hilm\) be a
  closed and purely outer \(S\)\nb-action on~\(A\) such that
  \(A\rtimes_\red S \congto B\).  We may identify~\(B\) with
  \(A\rtimes_\red S\) and assume that the chosen conditional
  expectation~\(E\) is the canonical one on \(A\rtimes_\red S\).
  Let \(\varphi\colon I\to A\rtimes_\red S\) be a virtual commutant.
  We must show that \(\varphi(I) \subseteq A\).  Let
  \(E^\bot = \Id - E\) be the projection to \(\ker E \subseteq B\).
  The claim \(\varphi(I) \subseteq A\) is equivalent to
  \(E^\bot\circ \varphi=0\).  Since \(E^\bot\circ \varphi\) is a
  virtual commutant as well, it suffices to prove \(\varphi=0\) for
  any virtual commutant \(\varphi\colon I \to \ker E \subseteq B\).

  Let \(t\in S\).  We will use the orthogonal projections
  \(E_t\colon \rg(B_t)\cdot (A\rtimes_\red S) \to B_t\)
  for \(t\in S\) defined in
  Proposition~\ref{pro:harmonic_projections}; here
  \(\rg(B_t) = B_t B_t^*\).  The map
  \[
    \varphi_t\defeq E_t\circ \varphi|_{\rg(B_t)\cap I}\colon
    \rg(B_t)\cap I \to B_t \subseteq A\rtimes_\red S
  \]
  is a virtual commutant as well.  Assume that \(\varphi_t\neq0\)
  for some \(t\in S\).  Let \(J \defeq \rg(B_t)\cap I\).
  Lemma~\ref{lem:virtual_commutant_multiplier} provides
  \(\tau_t\in \Mult(J B J) \cap J'\) with
  \(\varphi_t(x) = \tau_t \cdot x\) for all \(x\in J\).  Now
  \(\tau_t \cdot J = \varphi_t(J) \subseteq B_t\) implies
  \(\tau_t^* \tau_t\cdot J = J \cdot \tau_t^* \tau_t\cdot J
  \subseteq B_t^* B_t \subseteq A\).  Thus
  \(\tau_t^* \tau_t \in \Mult(J)\).  Since~\(\tau_t\) commutes
  with~\(J\), the multiplier~\(\tau_t^* \tau_t\) is central by
  Lemma~\ref{lem:centre_mult}.  It is positive and non-zero because
  \(\varphi_t\neq0\) by assumption.  So there is~\(\varepsilon\)
  with \(0 < \varepsilon< \norm{\tau_t^* \tau_t}\).  Let
  \(K\defeq (\tau_t^* \tau_t - \varepsilon)_+\cdot J\).  This is a
  two-sided ideal because~\(\tau_t^* \tau_t\) is central, and it is
  non-zero because \((\tau_t^* \tau_t - \varepsilon)_+\neq0\).  When
  we restrict~\(\tau_t^* \tau_t\) to a central multiplier of~\(K\),
  it becomes strictly positive, hence invertible.  So the formula
  \[
    \varphi_t^K(x)
    \defeq \varphi_t((\tau_t ^* \tau_t)^{-1/2}x)
    = \tau_t \cdot (\tau_t ^* \tau_t)^{-1/2}x
  \]
  well defines a virtual commutant
  \(\varphi_t^K\colon K\to B_t\subseteq A\rtimes_\red S\).  It
  preserves the \(K\)-valued inner products:
  \(\varphi_t^K(x)^* \varphi_t^K(y)=x^*y\), for \(x,y\in K\).  In
  particular, \(\varphi_t^K\) is isometric.  The range
  of~\(\varphi\) is contained in \(\ker E\), \(E_t\) is
  \(A\)\nb-linear and \(\ker E \cap B_t=B_t I_{1,t}^\bot\)
  by~\eqref{eq:decompose_Hilm_closed}.  Then it follows that the
  range of~\(\varphi_t\) and therefore also of~\(\varphi_t^K\) is
  contained in \(B_t I_{1,t}^\bot\).  Thus \(\varphi_t^K(K)\) is a
  Hilbert subbimodule of \(B_t I_{1,t}^\bot\cdot K\) which is
  isomorphic to~\(K\).  This is equivalent to
  \(B_t I_{1,t}^\bot\cdot K\cong K\) by
  Lemma~\ref{lem:Hilbert_module_outerness}.  So
  \(B_t\cdot I_{1,t}^\bot\) is not purely outer, a contradiction.
  This contradiction shows that the virtual commutants~\(\varphi_t\)
  are zero for all \(t\in S\).  Equivalently,
  \(E_t\bigl(\rg(B_t)\varphi(I)\bigr)=E_t\bigl(\varphi(I\cap
  \rg(B_t))\bigr)=0\) for all \(t\in S\).  Then \(\varphi(I)=0\)
  follows by the last part of
  Proposition~\ref{pro:harmonic_projections}.
\end{proof}


\section{Uniqueness of the crossed product decomposition}
\label{sec:uniqueness_Cartan}

An important feature of Renault's theory of Cartan subalgebras is
that the twisted groupoid obtained from the Cartan subalgebra
\(A\subseteq B\) is unique up to isomorphism.  So any automorphism
\(\beta\in\Aut(B)\) with \(\beta(A)=A\) lifts to an automorphism of
the underlying twisted groupoid.  In this article, étale groupoids are
replaced by inverse semigroup actions.  These are no longer unique
up to isomorphism because an étale groupoid~\(H\) may be written as
\(H \cong X\rtimes S\) for any wide, unital inverse
subsemigroup~\(S\) of \(\Bis(H)\), but only the action of
\(\Bis(H)\) on~\(X\) is canonically defined through the
groupoid~\(H\).  We are going to define a ``refinement'' for general
inverse semigroup actions on \(\Cst\)\nb-algebras, using either of
the dual groupoids \(\widecheck{A}\rtimes S\) or
\(\dual{A}\rtimes S\).  Their inverse semigroups of bisections
coincide:

\begin{lemma}
  \label{lem:dual_groupoids_bisections}
  Let~\(\Hilm\) be an inverse semigroup action of~\(S\) on~\(A\).
  The map
  \(\kappa\colon \dual{A}\rtimes S \to \widecheck{A}\rtimes S\),
  \([t,[\pi]]\mapsto [t,\ker \pi]\), is a continuous open
  epimorphism of groupoids.  It induces an isomorphism
  between the lattices of open subsets
  \(\Open(\dual{A}\rtimes S)\cong \Open(\widecheck{A}\rtimes S)\)
  and an isomorphism of inverse semigroups
  \(\Bis(\dual{A}\rtimes S)\cong \Bis(\widecheck{A}\rtimes S)\).
\end{lemma}

\begin{proof}
  It readily follows from the definitions that~\(\kappa\) is a
  groupoid epimorphism.  Let \(t\in S\) and let \(J\in\Ideals(A)\)
  be contained in the source ideal
  \(\s(\Hilm_t) = \braket{\Hilm_t}{\Hilm_t}\).  Then
  \(U_t \widecheck{J}\defeq \setgiven*{[t,\prid]}{\prid\in
    \widecheck{J}}\) is a bisection of~\(\widecheck{A}\rtimes S\)
  and
  \(W_t \dual{J}\defeq \setgiven*{[t,[\pi]]}{[\pi]\in \dual{J}}\) is
  a bisection of~\(\dual{A}\rtimes S\).  Such bisections form bases
  for the topology of~\(\widecheck{A}\rtimes S\)
  and~\(\dual{A}\rtimes S\), respectively.  Moreover,
  \(\kappa^{-1}(U_t\widecheck{J})=W_t \dual{J}\).  This implies
  that~\(\kappa\) is continuous open and induces an isomorphism
  \(\Open(\dual{A}\rtimes S)\cong \Open(\widecheck{A}\rtimes S)\).
  To see that this isomorphism restricts to
  \(\Bis(\dual{A}\rtimes S)\cong \Bis(\widecheck{A}\rtimes S)\) we
  describe the corresponding bisections more explicitly.

  Any bisection \(u\subseteq \widecheck{A}\rtimes S\) is a union of
  bisections of the form~\(U_t\widecheck{J}\) for some \(t\in S\),
  \(J\subseteq\s(\Hilm_t)\).  Consider such a union
  \(u\defeq \bigcup_{i\in I} U_{t_i}\widecheck{J}_i\).  It is a
  bisection if and only if both \(\rg\) and~\(\s\) are injective
  on~\(u\).  Recall that \([t_i,\prid]=[t_j,\prid]\) for
  \(\prid \in \widecheck{J}_i\cap \widecheck{J}_j\) holds if and
  only if there is \(v\in S\) with \(v \le t_i, t_j\) and
  \(\prid\in \widecheck{\s(\Hilm_v)}\).  And this is further
  equivalent to \(\prid\in \widecheck{I}_{t_i,t_j}\)
  with~\(I_{t_i,t_j}\) as in~\eqref{eq:Itu}.  Thus~\(\s\) is
  injective if and only if \(J_i \cap J_j \subseteq I_{t_i,t_j}\)
  for all \(i,j\in I\).  We
  rewrite the injectivity of~\(\rg\) on~\(u\) through the
  injectivity of~\(\s\) on
  \(u^{-1} \defeq \setgiven{\gamma^{-1}\in \widecheck{A}\rtimes
    S}{\gamma \in u}\).  As a result, a subset~\(u\)
  of~\(\widecheck{A}\rtimes S\) is a bisection if and only if both
  \(u\) and~\(u^{-1}\) are of the form
  \(\bigcup_{i\in I} U_{t_i}\widecheck{J}_i\) for a set~\(I\) and
  \(t_i \in S\), \(J_i \in\Ideals(A)\),
  \(J_i \subseteq \s(\Hilm_{t_i})\), such that
  \(J_i \cap J_j \subseteq I_{t_i,t_j}\) for all \(i,j\in I\).

  Similar arguments show that a subset~\(w\)
  of~\(\dual{A}\rtimes S\) is a bisection if and only if both \(w\)
  and~\(w^{-1}\) are of the form
  \(\bigcup_{i\in I} W_{t_i}\dual{J}_i\) for a set~\(I\) and
  \(t_i \in S\), \(J_i \in\Ideals(A)\),
  \(J_i \subseteq \s(\Hilm_{t_i})\), such that
  \(J_i \cap J_j \subseteq I_{t_i,t_j}\) for all \(i,j\in I\).
  Since
  \(\kappa( \bigcup_{i\in I} W_{t_i}\dual{J}_i)= \bigcup_{i\in I}
  U_{t_i}\widecheck{J}_i\), we conclude that \(\kappa\) induces a
  bijection
  \(\Bis(\dual{A}\rtimes S)\cong \Bis(\widecheck{A}\rtimes S)\).  It
  preserves the semigroup product because~\(\kappa\) is a groupoid
  homomorphism.  Since we shall need this later, we describe the
  multiplication in \(\Bis(\dual{A}\rtimes S)\) and
  \(\Bis(\widecheck{A}\rtimes S)\) more explicitly.  For each
  \(t\in S\), the Rieffel correspondence gives a lattice isomorphism
  \(\widecheck{\Hilm}_{t}\colon \Ideals(\s(\Hilm_t)) \to
  \Ideals(\rg(\Hilm_t))\) that restricts to the homeomorphism
  \(\widecheck{\Hilm}_{t}\colon \widecheck{\s(\Hilm_t)} \to
  \widecheck{\rg(\Hilm_t)}\) which corresponds to the bisection
  \(U_t \defeq \setgiven*{[t,\prid]}{\prid\in
    \widecheck{\s(\Hilm_t)}}\).  In particular, for any ideals
  \(J,J'\) in \(A\) the open set
  \(U_{t}^{-1}(\widecheck{J})\widecheck{J'}\) is the primitive ideal
  space of \(\widecheck{\Hilm}_{t}^{-1}(J) J'\).  Hence for any two
  open sets \(u=\bigcup_{i} U_{t_i}\widecheck{J}_i\) and
  \(u' = \bigcup_{j} U_{s_j}\widecheck{J}_j'\) we get
  \begin{equation}
    \label{eq:product_bisections_refinement}
    u u'
    = \bigcup_{i,j} U_{t_i}\widecheck{J}_i U_{s_j}\widecheck{J}_j'
    =\bigcup_{i,j} U_{t_is_j}
    U_{s_j}^{-1}(\widecheck{J}_i)\widecheck{J}_j'=\bigcup_{i,j} U_{t_is_j}\widecheck{J_{i,j}},
  \end{equation}
  where \(J_{i,j}\defeq \widecheck{\Hilm}_{s_j}^{-1}(J_i) J_j'\) for
  all \(i,j\).  Similarly, for \(w=\bigcup_{i} W_{t_i}\dual{J}_i\)
  and \(w' = \bigcup_{j} W_{s_j}\dual{J}_j'\) we get
  \(ww'=\bigcup_{i,j} W_{t_is_j}\dual{J_{i,j}}\).  This implies
  \(\kappa(ww')=\kappa(w)\kappa(w')\).
\end{proof}

\begin{definition}
  Let \(\Hilm_j = (\Hilm_{j,t})_{t\in S_j}\) for \(j=1,2\) be
  actions of unital inverse semigroups~\(S_j\) on
  \(\Cst\)\nb-algebras~\(A_j\).  An \emph{isomorphism} between these
  actions is given by a family of isomorphisms
  \(\varphi\colon A_1 \congto A_2\), \(\psi\colon S_1 \congto S_2\),
  and \(\varphi_t\colon \Hilm_{1,t} \congto \Hilm_{2,\psi(t)}\) for
  \(t\in S_1\), such that the isomorphisms~\(\varphi_t\) intertwine
  the multiplication isomorphisms
  \(\Hilm_{j,t}\otimes_{A_2} \Hilm_{j,u} \to \Hilm_{j,t u}\) for
  \(j=1,2\) and \(t,u\in S_j\).
\end{definition}

\begin{proposition}
  \label{prop:refine_action}
  Let~\(A\) be a \(\Cst\)\nb-algebra, let~\(S\) be a unital inverse
  semigroup, and let \(\Hilm=(\Hilm_t,\mu_{t,u})_{t,u\in S}\) be an
  \(S\)\nb-action on~\(A\).  Equip~\(\widecheck{A}\) with the dual
  action and form the transformation
  groupoid~\(\widecheck{A}\rtimes S\).  Consider the unital inverse
  semigroup
  \(\tilde{S} \defeq \Bis(\widecheck{A}\rtimes S)\cong
  \Bis(\dual{A}\rtimes S)\).
  \begin{enumerate}
  \item \label{en:refine_action_1}%
    There is an \(\tilde{S}\)\nb-action~\(\tilde{\Hilm}\) on~\(A\)
    such that the \(S\)\nb-action~\(\Hilm\) factors
    through~\(\tilde{\Hilm}\) via the inverse semigroup homomorphism
    \(S\ni t \mapsto U_t\defeq \setgiven*{[t,\prid]}{\prid\in
      \widecheck{\s(\Hilm_t)}}\in \tilde{S}\) and such that
    \(\s(\tilde{\Hilm}_u) = \bigcap_{\prid\in \s(u)} \prid\) for all
    \(u\in \tilde{S}\); here \(\s(\tilde{\Hilm}_u) =\s(u)\) if we
    identify open sets in~\(\widecheck{A}\) with ideals
    in~\(\Ideals(A)\).  And any two such \(\tilde{S}\)\nb-actions
    are isomorphic through a unique isomorphism.
  \item \label{en:refine_action_2}%
    The dual groupoids \(\dual{A}\rtimes \tilde{S}\) and
    \(\widecheck{A}\rtimes \tilde{S}\) for the action
    \(\tilde{\Hilm}\) are naturally isomorphic to the corresponding
    dual groupoids~\(\dual{A}\rtimes S\) and
    \(\widecheck{A}\rtimes S\) for the action \(\Hilm\).
  \item \label{en:refine_action_3}%
    There is a canonical isomorphism
    \(A\rtimes_{\Hilm} S \cong A\rtimes_{\tilde{\Hilm}} \tilde{S}\),
    and this isomorphism descends to an isomorphism
    \(A\rtimes_{\red,\Hilm} S \cong A\rtimes_{\red,\tilde{\Hilm}}
    \tilde{S}\).
    \end{enumerate}
\end{proposition}

\begin{proof}
  Let~\(B\) be an exotic crossed product, that is, a
  \(\Cst\)\nb-algebra between \(A\rtimes S\) and
  \(A\rtimes_\red S\).  View~\(\Hilm_{t_i}\) for \(i\in I\) as a
  slice for the inclusion \(A \hookrightarrow B\).  Let
  \(u\in \tilde{S}\).  The description of elements in~\(\tilde{S}\)
  in the proof of Lemma~\ref{lem:dual_groupoids_bisections} shows
  that \(u = \bigcup_{i\in I} U_{t_i}\widecheck{J}_i\) for some
  \(t_i \in S\), \(J_i \in\Ideals(A)\),
  \(J_i \subseteq \s(\Hilm_{t_i})\) with
  \(J_i \cap J_j \subseteq I_{t_i,t_j}\) for all \(i,j\in I\).  We
  claim that
  \begin{equation}\label{eq:refined_slice}
    \tilde{\Hilm}_u
    \defeq \overline{\sum_{i\in I} \Hilm_{t_i}\cdot J_i}
    \subseteq B
  \end{equation}
  is a slice.  As a closed linear span of closed
  \(A\)\nb-subbimodules, it is again a closed \(A\)\nb-subbimodule.
  We must show that
  \(\tilde{\Hilm}_u \cdot \tilde{\Hilm}_u^*\subseteq A\) and
  \(\tilde{\Hilm}_u^* \cdot \tilde{\Hilm}_u\subseteq A\).  The
  Hilbert bimodules \(\Hilm_{t_i}\cdot I_{t_i,t_j}\) and
  \(\Hilm_{t_j}\cdot I_{t_i,t_j}\) are mapped to the same subspace
  in~\(B\).  Since \(J_i \cap J_j \subseteq I_{t_i,t_j}\) for all
  \(i,j\in I\), it follows that
  \[
    \Hilm_{t_i}\cdot J_i \cdot (\Hilm_{t_j}\cdot J_j)^*
    = \Hilm_{t_i}\cdot (J_i \cap J_j) \cdot \Hilm_{t_j}^*
    \subseteq \Hilm_{t_i}\cdot I_{t_i, t_j} \cdot \Hilm_{t_j}^*
    \subseteq \Hilm_{t_j} \Hilm_{t_j}^*
    \subseteq A.
  \]
  This implies
  \(\tilde{\Hilm}_u \cdot \tilde{\Hilm}_u^*\subseteq A\).  The other
  condition \(\tilde{\Hilm}_u^* \cdot \tilde{\Hilm}_u\subseteq A\)
  follows from the same argument applied to~\(u^{-1}\), which has
  the same structure as~\(u\).

  Hence \eqref{eq:refined_slice} defines a map
  \(\tilde{S}\ni u\mapsto \tilde{\Hilm}_u\in \Slice(A,B)\).  This
  map preserves the involution because if
  \(u = \bigcup_{i\in I} U_{t_i}\widecheck{J}_i\) then
  \(u^{-1} = \bigcup_{i\in I} \widecheck{J}_iU_{t_i^*}=\bigcup_{i\in
    I} U_{t_i^*} \Prim(\widecheck{\Hilm}_{t_i}(J))\), where
  \(\widecheck{\Hilm}_{t}\colon \Ideals(\s(\Hilm_t)) \to
  \Ideals(\rg(\Hilm_t))\) is the Rieffel isomorphism, and therefore
  \[
    \tilde{\Hilm}_u ^*
    = \overline{\sum_{i} J_i\Hilm_{t_i^*}}
    = \overline{\sum_{i} \Hilm_{t_i^*}\widecheck{\Hilm}_{t_i}(J)}
    = \tilde{\Hilm}_{u^{-1}}.
  \]
  The map \(u\mapsto \tilde{\Hilm}_u\) is a semigroup homomorphism
  because if \(u' = \bigcup_{j} U_{s_j}\widecheck{J}_j'\) is another
  element in~\(\tilde{S}\), in a canonical form,
  then~\eqref{eq:product_bisections_refinement} implies
  \[
    \tilde{\Hilm}_u \tilde{\Hilm}_{u'}
    = \overline{\sum_{i,j} \Hilm_{t_i}J_i\Hilm_{s_j}J_j'}
    = \overline{\sum_{i,j} \Hilm_{t_is_j}
      \widecheck{\Hilm}_{s_j}(J_i) J_j'}
    = \tilde{\Hilm}_{uu'}.
  \]
  The multiplication in~\(B\) restricts to maps
  \(\mu_{t,u} \colon \tilde{\Hilm}_t \otimes_A \tilde{\Hilm}_u
  \congto \tilde{\Hilm}_{t u}\).  By construction,
  \(\tilde{\Hilm}_{U_t} = \Hilm_t\) for all \(t\in S\).  Thus the
  action~\(\Hilm\) factors through an action~\(\tilde{\Hilm}\)
  of~\(\tilde{S}\) with the required properties.

  Now let \((\tilde{\Hilm}'_u,\mu_{t,u}')_{t,u\in\tilde{S}}\) be any
  action of~\(\tilde{S}\) through which the given action~\(\Hilm\)
  factors; this means that \(\tilde{\Hilm}'_{U_t} = \Hilm_t\) and
  \(\mu_{U_t,U_u}' = \mu_{t,u}\) for all \(t,u \in S\).  In
  addition, we assume that \(\s(\tilde{\Hilm}'_u) = \s(u)\) for all
  \(u\in \tilde{S}\) (we identify the lattice of open subsets
  in~\(\widecheck{A}\) with \(\Ideals(A)\)).  Any action
  of~\(\tilde{S}\) comes with canonical inclusion maps
  \(\tilde{\Hilm}'_t \hookrightarrow\tilde{\Hilm}'_u\) for
  \(t, u \in \tilde{S}\) with \(t \le u\).  Write \(u\in \tilde{S}\)
  as \(u = \bigcup_{i\in I} U_{t_i}\widecheck{J_i}\).  Then
  \(U_{t_i} \ge U_{t_i}\widecheck{J_i}\le u\) for all \(i\in I\).
  The inclusion
  \(\tilde{\Hilm}'_{U_{t_i}|_{J_i}} \hookrightarrow
  \tilde{\Hilm}'_{U_{t_i}} = \Hilm_{t_i}\) is an isomorphism onto
  \(\Hilm_{t_i}\cdot J_i\).  So the inclusion maps above yield
  canonical inclusion maps
  \(\Hilm_{t_i}\cdot J_i \hookrightarrow \tilde{\Hilm}'_u\) for all
  \(i\in I\).  The resulting map
  \(\bigoplus_{i\in I} \Hilm_{t_i}\cdot J_i \to \tilde{\Hilm}'_u\)
  has dense range because its image is an \(A\)\nb-subbimodule~\(V\)
  such that~\(\braket{V}{V}\) is dense in
  \(\braket{\tilde{\Hilm}'_u}{\tilde{\Hilm}'_u} = \s(u)\).  The
  Hilbert \(A\)\nb-bimodule~\(\tilde{\Hilm}_u\) is isomorphic to the
  Hausdorff completion of
  \(\bigoplus_{i\in I} \Hilm_{t_i}\cdot J_i\) for the semi-norm
  given by the inner product
  \(\braket{(x_i)}{(y_i)} = \sum_{i,j\in I} x_i^* y_j\).  The norm
  in~\(\tilde{\Hilm}'_u\) is given by the same formula.  Hence the
  map
  \(\bigoplus_{i\in I} \Hilm_{t_i}\cdot J_i \to \tilde{\Hilm}'_u\)
  descends to an isomorphism
  \(\tilde{\Hilm}_u \congto \tilde{\Hilm}'_u\).  These isomorphisms
  for \(u\in\tilde{S}\) are clearly compatible with the
  multiplication maps and thus define an isomorphism of
  \(\tilde{S}\)\nb-actions.  Since any isomorphism of
  \(\tilde{S}\)\nb-actions is compatible with the inclusion maps for
  \(t \le u\) in~\(\tilde{S}\), this is the only isomorphism of
  \(\tilde{S}\)\nb-actions that is the identity map on
  \(\tilde{\Hilm}_{U_t} = \Hilm_t\) for all \(t\in S\).  This
  proves~\ref{en:refine_action_1}.

  The dual \(\tilde{S}\)\nb-action of~\(\tilde{\Hilm}\)
  on~\(\widecheck{A}\) turns out to be the standard action of
  bisections of \(\widecheck{A}\rtimes S\) on~\(\widecheck{A}\).  So
  \(\widecheck{A}\rtimes \tilde{S}\cong\widecheck{A}\rtimes S\)
  follows from \cite{Kwasniewski-Meyer:Essential}*{Proposition~2.2},
  applied to \(\Bis(\widecheck{A}\rtimes S)\) itself.  The above
  argument with~\(\tilde{S}\) written as
  \(\Bis(\dual{A}\rtimes S)\) shows
  \(\dual{A}\rtimes \tilde{S}\cong\dual{A}\rtimes S\).  This
  proves~\ref{en:refine_action_2}.

  We prove~\ref{en:refine_action_3}.  The homomorphism
  \(\tilde{S} \ni u\mapsto \tilde{\Hilm}_u\in \Slice(A,A\rtimes S)\)
  given by~\eqref{eq:refined_slice} is a representation
  of~\(\tilde{\Hilm}\) in \(A\rtimes S\) and thus induces a
  \Star{}homomorphism
  \(\tilde{\Phi}\colon A\rtimes \tilde{S} \to A\rtimes S\).  As
  \(\tilde{\Hilm}\) factors through~\(\Hilm\), there is also a
  representation of~\(\Hilm\) in \(A\rtimes \tilde{S}\), which
  induces a homomorphism
  \(\Phi\colon A\rtimes S \to A\rtimes \tilde{S}\) with
  \(\tilde{\Phi}\circ \Phi|_{\Hilm_t}=\Id_{\Hilm_t}\) for
  \(t\in S\).  The images of~\(\Hilm_t\) for \(t\in S\) are linearly
  dense both in \(A\rtimes S\) and \(A\rtimes \tilde{S}\) because
  any~\(\Hilm_u\) for \(u\in\tilde{S}\) is the closed linear span
  of~\(\Hilm_t\) for \(t\in S\) with \(t \le u\).  The maps \(\Phi\)
  and~\(\tilde{\Phi}\) restrict to the identity maps between the
  images of~\(\Hilm_t\) in \(A\rtimes S\) and
  \(A\rtimes \tilde{S}\).  Hence \(\Phi\) and~\(\tilde\Phi\) are
  isomorphisms inverse to each other.  And the canonical weak
  conditional expectation on \(A\rtimes \tilde{S}\) is determined by
  its restrictions to the slices \(\tilde{\Hilm}_{U_t} = \Hilm_t\)
  for \(t\in S\).  These restrictions are given by the same formula
  as the canonical weak conditional expectation on \(A\rtimes S\)
  (see Proposition~\ref{prop:conditional_expectation}).  So the
  isomorphism \(A\rtimes \tilde{S} \cong A\rtimes S\) intertwines
  the weak conditional expectations.  Then it descends to an
  isomorphism \(A\rtimes_\red \tilde{S} \cong A\rtimes_\red S\).
  This proves the statement~\ref{en:refine_action_3}.
\end{proof}

\begin{definition}
  Let~\(A\) be a \(\Cst\)\nb-algebra, let~\(S\) be a unital inverse
  semigroup, and let~\(\Hilm\) be an \(S\)\nb-action on~\(A\).  The
  essentially unique action of~\(\tilde{\Hilm}\) on~\(A\) described
  in Proposition~\ref{prop:refine_action} is called the \emph{refinement}
  of~\(\Hilm\).
\end{definition}

The following example shows that for a general inverse semigroup
action~\(\Hilm\), the inverse semigroup \(\Slice(A,B)\) may be much
larger than~\(\tilde{S}\).  Then the crossed product
\(A\rtimes_\red\Slice(A,B)\) and the dual groupoids
\(\dual{A}\rtimes \Slice(A,B)\) and
\(\widecheck{A}\rtimes \Slice(A,B)\) are much larger than the
corresponding objects for \(S\) or~\(\tilde{S}\).

\begin{example}
  Let \(S=\Z\sqcup\{0\}\) be obtained by adding a zero element to
  the group~\(\Z\) and let~\(S\) act trivially on~\(\C\).  This
  action is its own refinement because \(\widecheck{\C}\rtimes S\)
  is isomorphic to the
  group~\(\Z\) and~\(S\) is its inverse semigroup of bisections.  In
  contrast, a slice for \(\C\subseteq \Cst(\Z)\) is either~\(0\) or
  of the form \(u\cdot \C\) for a unitary in \(\Cst(\Z)\), which is
  unique up to multiplication with a scalar \(\lambda\in\C\) with
  \(\abs{\lambda}=1\).  So in this example, \(\tilde{S}\) is much
  smaller than \(\Slice(\C,\Cst(\Z))\).  And
  \(\C\rtimes \tilde{S} = \Cst(\Z)\) is remarkable because
  \(\C\rtimes \Slice(\C,\Cst(\Z))\) is not even separable.
\end{example}

\begin{theorem}
  \label{thm:uniqueness_purely_outer_action}
  Let~\(\Hilm\) be a closed and purely outer action of a unital
  inverse semigroup~\(S\) on a \(\Cst\)\nb-algebra~\(A\).  Let
  \(B\defeq A\rtimes_\red S\).  The refined action~\(\tilde{\Hilm}\)
  is canonically isomorphic to the tautological action of
  \(\Slice(A,B)\) on~\(A\).  And
  \[
    A\rtimes_\red\Slice(A,B)\cong B,\qquad
    \dual{A}\rtimes \Slice(A,B)\cong \dual{A}\rtimes S,\qquad
    \widecheck{A}\rtimes \Slice(A,B)\cong \widecheck{A}\rtimes S.
  \]
\end{theorem}

\begin{proof}
  The proof of Proposition~\ref{prop:refine_action} shows that we
  may treat~\(\tilde{\Hilm}_u\) for \(u\in \tilde{S}\) as elements
  of \(\Slice(A,B)\) and that the map
  \(\psi\colon \tilde{S} \to \Slice(A,B)\),
  \(u\mapsto \tilde{\Hilm}_u\), is a homomorphism.  The
  assertion follows once we show that~\(\psi\) is an isomorphism
  --~the second part then follows from
  Proposition~\ref{prop:refine_action}.
  So let \(X\in \Slice(A,B)\).  It
  suffices to prove that there is a unique \(t\in\tilde{S}\) with
  \(X=\tilde{\Hilm}_t\), where~\(\tilde{\Hilm}_t\) is defined
  in~\eqref{eq:refined_slice}.  Let
  \[
    T = \setgiven{U_t\widecheck{J}}
    {J\in\Ideals(A),\ \Hilm_t\cdot J \subseteq X,\ t\in S} \subseteq \tilde{S}.
  \]
  We claim that the union \(\bigcup T\) is a bisection
  of~\(\widecheck{A}\rtimes S\).  Let
  \(U_t\widecheck{J}, U_w\widecheck{K} \in T\).  The products
  \((\Hilm_t J)^*\cdot (\Hilm_w K)\) and
  \((\Hilm_t J)\cdot (\Hilm_w K)^*\) in~\(B\) are contained in~\(A\)
  because \(X^* X \subseteq A\) and \(X X^* \subseteq A\).  Since
  \((\Hilm_t J)\cdot (\Hilm_w K)^*=\Hilm_t (J\cap K) \Hilm_{w^*}
  \subseteq \Hilm_{t w^*}\) we get
  \(\Hilm_t (J\cap K) \Hilm_{w^*}\subseteq \Hilm_{t w^*} \cap A =
  I_{t w^*,1}\).  This implies that any
  \(\prid \in \widecheck{\rg(\Hilm_w)}\) with
  \(\widecheck{\Hilm}_{w^*}(\prid) \in \widecheck{J} \cap
  \widecheck{K}\) already belongs to \(\widecheck{I}_{t w^*,1}\).
  And this is equivalent to \(J \cap K \subseteq I_{t,w}\) and then
  to~\(\s\) being injective on
  \(U_t\widecheck{J}\cup U_w\widecheck{K} \subseteq \dual{A}\rtimes
  S\).  Similarly, \((\Hilm_t J)\cdot (\Hilm_w K)^* \subseteq A\) is
  equivalent to~\(\rg\) being injective on
  \(U_t\widecheck{J}\cup U_w\widecheck{K} \subseteq \dual{A}\rtimes
  S\).  This implies that~\(\bigcup T\) is a bisection
  of~\(\dual{A}\rtimes S\).

  Let \(X_1\defeq \tilde{\Hilm}_{\bigcup T}\).  This is the closed
  linear span in~\(B\) of~\(\tilde{\Hilm}_t\) for \(t\in T\).  By
  construction, \(X_1\subseteq X\).  We have described~\(\tilde{S}\)
  above through unions of bisections of the form
  \(U_t\widecheck{J}\) for \(J\in\Ideals(\s(\Hilm_t))\), \(t\in S\).
  As a consequence, \(\bigcup T\) is the maximal element \(u\) in
  \(\tilde{S}\) with the property that
  \(\tilde{\Hilm}_u \subseteq X\).  Since any \(u\in\tilde{S}\) with
  \(u < \bigcup T\) satisfies \(\s(u) \subsetneq \s(\bigcup T)\), it
  follows that \(\tilde{\Hilm}_u \neq X\) for \(u\neq\bigcup T\).
  So the proof (injectivity and surjectivity of~\(\psi\)) is
  finished when we show that \(X_1 = X\).  Assume the contrary.
  Then~\(X_1\) is a proper sub-bimodule of the Hilbert
  bimodule~\(X\).  The Rieffel correspondence shows that the map
  \(Y\mapsto \s(Y)\) is a bijection between closed
  sub-bimodules in~\(X\) and ideals in~\(s(X)\).  Hence
  \(\s(X_1) \subsetneq \s(X)\).  So there is
  \(\prid \in \widecheck{\s(X)} \setminus \widecheck{\s(X_1)}\).
  Then there is \(x\in X\) with \((x^* x)_{\prid} \neq0\)
  in~\(A/\prid\).  Let
  \(\varepsilon \defeq \norm{x^* x}_{\prid}/2>0\).  There is
  \(y\in A\rtimes_\alg S\) with
  \(\norm{x} \cdot \norm{x-y}<\varepsilon\).  Let
  \(E\colon B \to A\) be the canonical conditional expectation.
  Then \(E(x^* x) = x^* x\) since~\(X\) is a slice.  Then
  \(E(x^* y)_{\prid} \neq0\) in~\(A/\prid\) because
  \(\norm{E(x^* y) - x^* x}_{\prid} = \norm{E(x^* y - x^*
    x)}_{\prid} <\varepsilon\).  Write \(y = \sum_{t\in F} y_t\) for
  a finite subset \(F\subseteq S\) and \(y_t\in \Hilm_t\) for
  \(t\in F\).  There must be \(t\in F\) with
  \(E(x^* y_t)_{\prid} \neq0\) in~\(A/\prid\).  We have
  \(x^* y_t \in X^* \Hilm_t\), and \(X^* \Hilm_t\) is a slice
  because it is a product of slices.  The inclusion \(A\subseteq B\)
  is a noncommutative Cartan subalgebra by
  Theorem~\ref{the:nc_Cartan}.
  Lemma~\ref{lem:expectation_unique_from_no_commutants} shows
  that~\(E\) maps the slice~\(X^*\Hilm_t\) into itself.  Then
  Proposition~\ref{pro:recognise_reduced_crossed_S_using_E} implies
  that \(X^* \Hilm_t = J \oplus Y\) with \(J = X^* \Hilm_t \cap A\)
  and \(Y\subseteq \ker E\).  Since \(E(x^* y_t)_{\prid}\neq0\), it
  follows that \(\prid \in \widecheck{J}\).  Now
  \(J\subseteq X^* \Hilm_t\) implies that
  \(X\cdot J = \Hilm_t \cdot J\) as slices in~\(B\).  So
  \(\Hilm_t \cdot J \subseteq X_1\).  Then
  \(\prid \in \widecheck{\s(X_1)}\), which contradicts our
  assumption.
\end{proof}

\begin{corollary}
  \label{cor:dual_groupoid_well-defined}
  Let \(A \subseteq B\) be a noncommutative Cartan subalgebra and
  let \((A,S_j,\Hilm_j)\) for \(j=1,2\) be two inverse semigroup
  actions for which there is an isomorphism
  \(A\rtimes_\red S_j \cong B\) mapping
  \(A\subseteq A\rtimes_\red S_j\) onto \(A\subseteq B\).  Then
  \(\tilde{\Hilm}_1\cong \tilde{\Hilm}_2\),
  \(\widecheck{A}\rtimes S_1\cong \widecheck{A}\rtimes S_2\) and
  \(\dual{A}\rtimes S_1\cong \dual{A}\rtimes S_2\).
\end{corollary}

\begin{proof}
  The two actions are closed and purely outer by
  Theorem~\ref{the:nc_Cartan}.  By
  Theorem~\ref{thm:uniqueness_purely_outer_action}, they have
  isomorphic refinements and isomorphic dual groupoids.
\end{proof}

The following example shows that the uniqueness result in
Corollary~\ref{cor:dual_groupoid_well-defined} fails for actions
that are not purely outer.

\begin{example}
  \label{exm:different_dual_groupoids}
  The groups \(\Z/2 \times \Z/2\) and~\(\Z/4\) are not isomorphic,
  but their group \(\Cst\)\nb-algebras are isomorphic to the algebra
  of functions on the discrete set with four points.  The inclusion
  of the unit does not add extra information.  So there are two
  essentially different ways to realise the unital inclusion
  \(\C\subseteq \C^4\) as the unital inclusion
  \(\C \subseteq \Cst_\red(H)\) for a discrete group~\(H\).  The
  resulting dual groupoids are the two groups \(\Z/2 \times \Z/2\)
  and~\(\Z/4\).
\end{example}

\section{Aperiodic inclusions versus Cartan \texorpdfstring{$\Cst$}{C*}-subalgebras}
\label{sec:pure_outer_vs_aperiodic}

Let \(A\subseteq B\) be a \(\Cst\)-inclusion.  We say that \(A\)
\emph{detects ideals} in \(B\) if \(J\cap A=0\) for an ideal~\(J\)
in \(B\) implies \(J=0\).  Theorem~\ref{the:nc_Cartan} implies that
a general non-commutative Cartan \(\Cst\)\nb-subalgebra
\(A\subseteq B\) does not detect ideals in~\(B\).  Indeed, there are
counterexamples where \(B=A\rtimes_\alpha \Z\) is the crossed
product by a single automorphism~\(\alpha\) on a separable
\(\Cst\)\nb-algebra.  An example of a purely outer
automorphism~\(\alpha\) for which the inclusion
\(A \hookrightarrow A\rtimes_\alpha \Z\) does not detect ideals is
described in \cite{Kwasniewski-Meyer:Aperiodicity}*{Examples 2.14
  and~2.21}.  For such automorphisms, \(A\) is a Cartan subalgebra
of \(A\rtimes_\alpha\Z\), but~\(A\) does not detect ideals in
\(A\rtimes_\alpha\Z\) (see
\cite{Kwasniewski-Meyer:Aperiodicity}*{Theorem~9.12}).  This means
that there is a non-trivial quotient \((A\rtimes_\alpha\Z)/J\) into
which~\(A\) embeds.  This quotient cannot admit an \(A\)\nb-valued
conditional expectation because then the conditional expectation
\(A\rtimes_\alpha \Z\to A\) would not be unique, in contradiction to
Theorem~\ref{the:nc_Cartan}.  This is why the lack of ``uniqueness''
for such counterexamples is not seen in Exel's theory of
noncommutative Cartan subalgebras.

In this section we discuss conditions under which a noncommutative
Cartan \(\Cst\)\nb-subalgebra \(A\subseteq B\) does detect ideals,
and thus is a good tool to study the structure of~\(B\).
We compare Cartan inclusions with aperiodic inclusion and combine
the results above with some of the results on these inclusions
in~\cite{Kwasniewski-Meyer:Essential}.  If the
\(\Cst\)\nb-inclusion \(A\subseteq B\) is aperiodic and
\(E\colon B\to A\) is a faithful conditional expectation, then~\(A\)
detects ideals in~\(B\) and, even more, \(A\) \emph{supports}~\(B\)
in the sense that for every \(b\in B^+\setminus\{0\}\) there is
\(a\in A^+\setminus\{0\}\) with \(a \precsim b\) (see
\cite{Kwasniewski-Meyer:Essential}*{Theorem~1.1}).
Here~\(\precsim\) denotes the Cuntz preorder on the set~\(B^+\) of
positive elements in~\(B\).

\begin{definition}[\cites{Kwasniewski-Meyer:Aperiodicity,
    Kwasniewski-Meyer:Essential}]
  \label{def:aperiodicity}
  Let~\(X\) be a normed \(A\)\nb-bimodule.  An element \(x\in X\)
  \emph{satisfies Kishimoto's condition} if for each non-zero
  hereditary subalgebra \(D\subseteq A\) and \(\varepsilon>0\) there
  is \(a\in D^+\) with \(\norm{a}=1\) and
  \(\norm{a x a}<\varepsilon\).  We call~\(X\) aperiodic if all
  elements \(x\in X\) satisfy Kishimoto's condition.  A
  \emph{\(\Cst\)\nb-inclusion \(A\subseteq B\) is aperiodic} if the
  Banach \(A\)\nb-bimodule \(B/A\) is aperiodic.  An \emph{inverse
    semigroup action~\(\Hilm\) is aperiodic} if the Hilbert
  \(A\)\nb-bimodules \(\Hilm_t \cdot I_{1,t}^\bot\) with
  \(I_{1,t}^\bot =\setgiven{a\in A}{a\cdot I_{1,t}=0}\) are
  aperiodic for all \(t\in S\).
\end{definition}

\begin{proposition}
  \label{prop:aperiodic_implies_uniqueness}
  Let \(A\subseteq B\) be a \(\Cst\)\nb-inclusion with a conditional
  expectation \(E\colon B\to A\).  The inclusion \(A\subseteq B\) is
  aperiodic if and only if the \(A\)\nb-bimodule \( \ker E\) is
  aperiodic.  In this case, \(E\) is the only conditional
  expectation \(B\to A\), and~\(E|_{I B I}\) is the only conditional
  expectation \(I B I \to I\) for all \(I\in\Ideals(A)\).
\end{proposition}

\begin{proof}
  Since \(B=\ker E\oplus A\), there is a contractive
  \(A\)\nb-bimodule isomorphism \(\ker E \to B/A\) with bounded
  inverse.  This implies that \(\ker E\) is aperiodic if and only
  if~\(B/A\) is aperiodic.  Now let \(P\colon B\to A\subseteq B\) be
  a conditional expectation.  Assume~\(\ker E\) to be aperiodic, and
  let \(x\in \ker E\).  Then \(P(x)^* x \in \ker E\) and
  \(a \defeq P(P(x)^* x) = P(x)^* P(x) \ge0\) because~\(P\) is
  \(A\)\nb-bilinear.  Since~\(P\) is a bounded bimodule map, \(a\)
  inherits Kishimoto's condition from \(P(x)^* x \in \ker E\).  By
  \cite{Kwasniewski-Meyer:Essential}*{Lemma~5.10}, \(0\) is the only
  positive element of~\(A\) that satisfies Kishimoto's condition.
  So \(a=0\) and then \(P(x)=0\).  Both \(P\) and~\(E\) are
  idempotent maps on~\(B\) with the same image~\(A\).  We have shown
  that \(\ker E \subseteq \ker P\), and this implies \(P = E\).
  So~\(E\) is the only conditional expectation \(B\to A\).  If
  \(I\in\Ideals(A)\), then the inclusion \(I \subseteq I B I\)
  inherits aperiodicity from \(A\subseteq B\) by
  \cite{Kwasniewski-Meyer:Essential}*{Proposition~5.15}.  Therefore,
  the conditional expectation~\(E|_{I B I}\) is unique as well.
\end{proof}

\begin{theorem}
  \label{thm:characterisation_of_aperiodic_crossed_products}
  Let \(A\subseteq B\) be a \(\Cst\)\nb-inclusion and consider the
  following conditions:
  \begin{enumerate}
  \item \label{enu:aperiodic_crossed_products1}%
    \(A\subseteq B\) is regular and aperiodic, and there is a
    faithful conditional expectation \(E\colon B\to A\);

  \item \label{enu:aperiodic_crossed_products2}%
    \(B\cong A\rtimes_\red S\) for a closed and aperiodic
    action~\(\Hilm\) of an inverse semigroup~\(S\) on~\(A\), with an
    isomorphism that restricts to the canonical embedding on~\(A\);
  \item \label{enu:noncommutative_Cartans333}%
    \(A\subseteq B\) is a noncommutative Cartan subalgebra.
  \end{enumerate}
  Then \ref{enu:aperiodic_crossed_products1}\(\Leftrightarrow\)%
  \ref{enu:aperiodic_crossed_products2}\(\Rightarrow\)%
  \ref{enu:noncommutative_Cartans333}, and all conditions
  \ref{enu:aperiodic_crossed_products1}--\ref{enu:noncommutative_Cartans333}
  are equivalent if~\(A\) is prime or contains an essential ideal of
  Type~I.

  Assume the equivalent conditions
  \ref{enu:aperiodic_crossed_products1}
  and~\ref{enu:aperiodic_crossed_products2}.  Then~\(A\) detects
  ideals in~\(B\) and~\(A\) supports~\(B\).  And~\(B\) is simple if
  and only if~\(\Hilm\) is minimal, that is, there are no
  non-trivial ideals in~\(A\) that are \(\Hilm_t\)-invariant for all
  \(t\in S\).  If~\(B\) is simple, then it is purely infinite if and
  only if all elements of \(A^+\setminus\{0\}\) are infinite
  in~\(B\).
\end{theorem}

\begin{proof}
  By our main theorem and
  \cite{Kwasniewski-Meyer:Essential}*{Theorem~6.14},
  \ref{enu:aperiodic_crossed_products2} implies
  \ref{enu:noncommutative_Cartans333} and the converse holds
  when~\(A\) is prime or contains an essential ideal of Type~I.
  \ref{enu:aperiodic_crossed_products2} implies
  \ref{enu:aperiodic_crossed_products1} by
  \cite{Kwasniewski-Meyer:Essential}*{Proposition 6.3}.  Now assume
  \ref{enu:aperiodic_crossed_products1}.  The last part of
  Proposition~\ref{prop:aperiodic_implies_uniqueness} verifies
  \ref{the:nc_Cartan0} in Theorem~\ref{the:nc_Cartan}.  Hence
  \(A\subseteq B\) is a noncommutative Cartan subalgebra, and
  \(B\cong A\rtimes_\red S\) for a closed action~\(\Hilm\) of an
  inverse semigroup~\(S\) on~\(A\).  Under this isomorphism
  \(E\colon B\to A\) is the canonical conditional expectation.
  Hence~\(\Hilm\) is aperiodic by
  \cite{Kwasniewski-Meyer:Essential}*{Proposition 6.3}.

 Thus \ref{enu:aperiodic_crossed_products1} and \ref{enu:aperiodic_crossed_products2} are equivalent.
 The last part of the assertion  follows from  \cite{Kwasniewski-Meyer:Essential}*{Theorems 6.5, 6.6 and Corollary 6.7}.
   \end{proof}
  \begin{corollary}
  \label{cor:cartan_detection_supporting}
  Let \(A\subseteq B\) be a noncommutative Cartan subalgebra such
  that~\(A\) is either prime or  contains an essential ideal
  of Type~I.  Then~\(A\) supports~\(B\) and thus~\(A\) detects
  ideals in~\(B\).
\end{corollary}

Theorem \ref{thm:characterisation_of_aperiodic_crossed_products}
indicates that regular aperiodic inclusions \(A\subseteq B\) with a
faithful conditional expectation \(E\colon B\to A\) form an
important subclass of noncommutative Cartan inclusions.
The example mentioned in the beginning of this section shows that
this subclass is strictly smaller --~unless~\(A\) is prime or contains
an essential ideal of Type~I.  If the action is only purely outer,
then we may recover the dynamics from the inclusion
\(A\subseteq A\rtimes_\red S\) in an essentially unique way; but we
cannot yet relate the ideal structure or the Cuntz semigroups of~\(A\)
and~\(A\rtimes_\red S\).  So some applications will only work for
aperiodic noncommutative Cartan inclusions.

Aperiodic noncommutative Cartan inclusions can often be
characterised by the topological freeness (effectivity) of the dual
groupoid:

\begin{definition}
  \label{def:topol_free_groupoids}
  An étale groupoid \(H\), possibly with non-Hausdorff unit
  space~\(X\), is \emph{effective} if any open bisection
  \(U\subseteq H\) with \(\rg|_U = \s|_U\) is contained in~\(X\).
\end{definition}

\begin{remark}
  In \cite{Kwasniewski-Meyer:Essential}*{Definition~2.19}, an étale
  groupoid~\(H\) is called \emph{topologically free} if there is no
  non-empty open bisection \(U\subseteq H\setminus X\) with
  \(\rg|_U = \s|_U\).  Topological freeness is weaker than
effectivity in general, but the two notions coincide if~\(X\) is
closed in~\(H\) (so for instance when \(H\) is Hausdorff).
\end{remark}

\begin{theorem}
  \label{thm:characterisation_of_aperiodic_crossed_products_via_topological_freeness}
  Let \(A\subseteq B\) be a \(\Cst\)\nb-inclusion with a
  faithful conditional expectation.  Assume that~\(A\) has an
  essential ideal which is separable or of Type~I.  The conditions
  \ref{enu:aperiodic_crossed_products1}
  and~\ref{enu:aperiodic_crossed_products2} in
  Theorem~\textup{\ref{thm:characterisation_of_aperiodic_crossed_products}}
  are equivalent to each of the following:
  \begin{enumerate}
  \item \label{enu:aperiodic_crossed_products3}%
    \(B\cong A\rtimes_\red S\), by an \(A\)\nb-preserving
    isomorphism, for a closed action~\(\Hilm\) of an inverse
    semigroup~\(S\) on~\(A\) whose dual groupoid
    \(\widehat{A}\rtimes S\) is effective;

  \item \label{enu:aperiodic_crossed_products4.5}%
    \(A\subseteq B\) is regular and the dual groupoid
    \(\widehat{A}\rtimes \Slice(A,B)\) is effective and has a closed
    space of units;

  \item \label{enu:aperiodic_crossed_products5}%
    there is a saturated, wide grading \((B_t)_{t\in S}\) on~\(B\)
    with unit fibre~\(A\), whose dual groupoid
    \(\widehat{A}\rtimes S\) is effective and has a closed space of
    units.
  \end{enumerate}
  If these conditions hold, then the dual groupoids
  \(\widehat{A}\rtimes \Slice(A,B)\) and \(\widehat{A}\rtimes S\)
  above are canonically isomorphic to each other.
\end{theorem}

\begin{proof}
  By \cite{Kwasniewski-Meyer:Essential}*{Theorem~6.14}, the first
  two conditions in
  Theorem~\ref{thm:characterisation_of_aperiodic_crossed_products}
  are equivalent to~\ref{enu:aperiodic_crossed_products3} in the
  present assertion.  So~\ref{enu:aperiodic_crossed_products3}
  implies that \(A\subseteq B\) is a noncommutative Cartan
  subalgebra.  Together with
  Theorem~\ref{the:nc_Cartan}.\ref{the:nc_Cartan6} and
  Corollary~\ref{cor:dual_groupoid_well-defined}, it follows that
  the dual groupoid in~\ref{enu:aperiodic_crossed_products4.5} is
  canonically isomorphic to the dual groupoid for the
  action~\(\Hilm\).  These groupoids have closed unit spaces by
  Proposition~\ref{prop:conditional_expectation}.
  Thus~\ref{enu:aperiodic_crossed_products3}
  implies~\ref{enu:aperiodic_crossed_products4.5}.
  Condition~\ref{enu:aperiodic_crossed_products4.5}
  implies~\ref{enu:aperiodic_crossed_products5} by taking
  \(S = \Slice(A,B)\).

  Assume~\ref{enu:aperiodic_crossed_products5} and let~\(\Hilm\) be
  the action coming from the grading \((B_t)_{t\in S}\).  By
  definition, \(\widehat{A}\rtimes S\) is the dual groupoid for the
  action \(\Hilm\).  Hence this action is closed and aperiodic (see
  Proposition~\ref{prop:conditional_expectation} and
  \cite{Kwasniewski-Meyer:Essential}*{Theorem~6.14}).  Thus the
  canonical conditional expectation \(E_0\colon A\rtimes S \to A\)
  is the unique conditional expectation from \(A\rtimes S\)
  onto~\(A\) by Proposition~\ref{prop:aperiodic_implies_uniqueness}.
  This implies that the canonical epimorphism
  \(\pi\colon A\rtimes S\to B\) intertwines \(E_0\) and
  \(E\colon B\to A\).  Then
  Proposition~\ref{pro:recognise_reduced_crossed_S_using_E} shows
  that~\(\pi\) descends to an isomorphism \(B\cong A\rtimes_\red S\)
  as in~\ref{enu:aperiodic_crossed_products3}.
\end{proof}

\section{Cartan \texorpdfstring{$\Cst$}{C*}-subalgebras with
Hausdorff primitive ideal space}
\label{sec:Hausdorff_primitive_ideal_space}

In this section, we assume that~\(A\) is a \(\Cst\)\nb-algebra with
Hausdorff primitive ideal space \(X\defeq \Prim(A)=\widecheck{A}\).
We are going to show that a regular inclusion with a conditional
expectation is Cartan if and only if the conditional expectation is
unique, and we shall rewrite the crossed product description in
Theorem~\ref{the:nc_Cartan} through a Fell bundle over a Hausdorff
groupoid.  We also consider the two extreme cases where~\(A\) is
either simple or commutative.

The Dauns--Hofmann isomorphism \(\Contb(X)\cong Z(\Mult(A))\) shows
that~\(A\) is a \(\Cont_0(X)\)-algebra in such a way that
\((fa)(x)=f(x)a(x)\) in \(A(x)\defeq A/\prid_x\) for all \(x\in X\),
\(f\in \Cont_0(X)\) and \(a\in A\).  An ideal~\(I\) in~\(A\)
corresponds to an open subset
\(\widecheck{I}\subseteq \widecheck{A}=X\).

\begin{proposition}
  \label{pro:unique_conditionals_commutative}
  Let \(A\subseteq B\) be a non-degenerate \(\Cst\)\nb-inclusion
  where~\(\widecheck{A}\) is Hausdorff.  If there is a unique
  conditional expectation \(E\colon B\to A\), then there is a unique
  conditional expectation \(I B I \to I\) for each
  \(I\in\Ideals(A)\).  Similar implications hold for unique faithful
  and unique almost faithful conditional expectations, respectively.
\end{proposition}

\begin{proof}
  Suppose that there are an ideal \(I\in\Ideals(A)\) and a
  conditional expectation \(P\colon I B I \to I\) that differs from
  the restriction of~\(E\).  Then there is \(b_0\in I B I\) with
  \(P(b_0) \neq E(b_0)\).  Then there is
  \(x\in \widecheck{I} \subseteq \widecheck{A}\) with
  \(P(b_0)(x) \neq E(b_0)(x)\).  Pick a function
  \(f\in\Cont_0(\widecheck{I})\) with \(f(x)\neq0\) and
  \(0\le f \le 1\).  If \(b\in B=ABA\), then
  \(f\cdot b\cdot f \in I B I\), so that \(P(f b f)\) is defined.
  We define
  \[
    \tilde{E} \colon B\to A,\qquad
    b \mapsto P(f b f) + (1-f^2)\cdot E(b).
  \]
  Since \(P\) and~\(E\) are completely positive, \(0 \le f \le 1\),
  and~\(f\in \Cont_0(\widecheck{I})\subseteq \Cont_0(X)\subseteq
  Z(\Mult(A))\), the map~\(\tilde{E}\) is positive and
  \(\tilde{E}(a) = f a f + (1-f^2) a = a\) for all \(a\in A\).  It
  follows that \(\norm{\tilde{E}}=1\) and so~\(\tilde{E}\) is a
  conditional expectation \(B\to A\).  And
  \[
    \tilde{E}(b_0)(x) - E(b_0)(x)
    = f^2(x) P(b_0)(x) -f^2(x) E(b_0)(x) \neq 0
  \]
  by construction.  So the conditional expectation~\(E\) is not
  unique.  If, in addition, \(E\) and~\(P\) are (almost) faithful,
  then so is~\(\tilde{E}\).
\end{proof}

\begin{theorem}
  \label{the:Cartan_Hausdorff}
  Let \(A\subseteq B\) be a regular \(\Cst\)\nb-subalgebra with a
  faithful conditional expectation \(E\colon B\to A\).
  Assume~\(\widecheck{A}\) to be Hausdorff.  The conditions in
  Theorem~\textup{\ref{the:nc_Cartan}}, which characterise Cartan
  subalgebras in the sense of Exel, are equivalent to the following
  conditions:
  \begin{enumerate}
  \item \label{the:Cartan_Hausdorff1}%
    there is no other conditional expectation \(B \to A\)
    besides~\(E\);
  \item \label{the:Cartan_Hausdorff2}%
    there is no other faithful conditional expectation
    \(B \to A\) besides~\(E\).
  \end{enumerate}
  If these conditions hold, then there is a continuous Fell bundle
  \(\A=(A_\gamma)_{\gamma\in H}\) over a Hausdorff, étale, locally
  compact groupoid~\(H\) with unit space~\(\widecheck{A}\) such that
  \(B\cong \Cst_\red(H,\A)\) with an isomorphism restricting to the
  canonical isomorphism \(A \cong\Cont_0(\widecheck{A}, \A)\).  This
  Fell bundle is unique up to isomorphism, and~\(H\) is isomorphic
  to the dual groupoid \(\widecheck{A}\rtimes S\) for any saturated,
  wide grading \((B_t)_{t\in S}\) on~\(B\) with unit fibre~\(A\).
\end{theorem}

\begin{proof}
  Conditions \ref{the:nc_Cartan0} and~\ref{the:nc_Cartan1} in
  Theorem~\ref{the:nc_Cartan} are equivalent to
  \ref{the:Cartan_Hausdorff1} and~\ref{the:Cartan_Hausdorff2} in
  this theorem by
  Proposition~\ref{pro:unique_conditionals_commutative}.  Assume
  these conditions.  By Theorem~\ref{the:nc_Cartan}, for any
  saturated, wide grading \((B_t)_{t\in S}\) on~\(B\) with unit
  fibre~\(A\), the action~\((B_t)_{t\in S}\) of~\(S\) on~\(A\) is
  closed and purely outer, and there is a canonical isomorphism
  \(A\rtimes_\red S\cong B\).
  Theorem~\ref{thm:uniqueness_purely_outer_action} shows that the
  dual groupoid \(\widecheck{A}\rtimes S\) is canonically isomorphic
  to the dual groupoid \(H=\widecheck{A}\rtimes \Slice(A,B)\), and
  \(\Bis(H)\cong \Slice(A,B)\) as inverse semigroups.  Hence
  \cite{Buss-Meyer:Actions_groupoids}*{Theorem~6.1} applied to the
  tautological \(\Slice(A,B)\)-grading on~\(B\) gives both existence
  and uniqueness of the desired Fell bundle
  \(\A=(A_\gamma)_{\gamma\in H}\).  The Fell bundle is a continuous
  field of Banach spaces because~\(A\) is a continuous field of
  \(\Cst\)\nb-algebras over~\(\widecheck{A}\) (see
  \cite{Kwasniewski-Meyer:Essential}*{Remark~7.17}).
\end{proof}

\subsection{Simple Cartan subalgebras}
\label{sec:Cartan_simple}

Assume~\(A\) to be simple.  The assumptions in
Theorem~\ref{the:nc_Cartan} involving ideals become empty for
\(I=\{0\}\), so that only the case \(I=A\) remains.  Any non-zero
slice \(M\subseteq B\) satisfies \(M^* M = A = M M^*\).  Hence
\(\Slice(A,B)\setminus\{0\}\) is a group, and any grading on~\(B\)
by an inverse semigroup~\(S\) with unit fibre~\(A\) may be
simplified to a saturated grading by a group by taking the image
of~\(S\) in \(\Slice(A,B)\) and discarding~\(0\).  Pure outerness
simplifies to outerness:

\begin{definition}[\cite{Kwasniewski-Meyer:Aperiodicity}]
  A Hilbert \(A\)\nb-bimodule~\(\Hilm[H]\) is \emph{outer} if it is
  not isomorphic to~\(A\) as a Hilbert bimodule.  A saturated Fell
  bundle \((B_t)_{t\in G}\) over a group~\(G\) with unit fibre~\(A\)
  is \emph{outer} if, for each \(t\in G\setminus\{1\}\), the Hilbert
  \(A\)\nb-bimodule~\(B_t\) is outer.
\end{definition}

\begin{corollary}
  \label{cor:simple_nc_Cartan}
  Let \(A\subseteq B\) be a regular \(\Cst\)\nb-subalgebra with a
  faithful conditional expectation \(E\colon B\to A\).
  Let~\(A\) be simple.  The following are equivalent:
  \begin{enumerate}
  \item \label{cor:nc_Cartan0}%
    \(A\subseteq B\) is a noncommutative Cartan subalgebra;
  \item \label{cor:nc_Cartan0.5}%
    \(A\subseteq B\) is an aperiodic inclusion;
  \item \label{cor:nc_Cartan1b}%
    there is at most one conditional expectation \(B\to A\);
  \item \label{cor:nc_Cartan1}%
    \(E\colon B\to A\) is the only faithful conditional
    expectation onto~\(A\);
  \item \label{cor:nc_Cartan2}%
    \(A' \cap \Mult(B) = \C\cdot 1\);
  \item \label{cor:nc_Cartan3}%
    any \(A\)-bimodule map \(\varphi\colon A\to B\) has range
    in~\(A\);
  \item \label{cor:nc_Cartan4}%
    if a slice \(M\subseteq B\) is isomorphic to~\(A\), then
    \(M=A\);
  \item \label{cor:nc_Cartan6}%
    for any saturated group grading \((B_t)_{t\in G}\) of~\(B\) with
    unit fibre~\(A\), the Fell bundle \(\B=(B_t)_{t\in G}\) is outer
    and saturated, and the canonical \Star{}homomorphism
    \(\pi\colon \Cst(\B)\to B\) descends to an isomorphism
    \(\Cst_\red(\B)\congto B\);
  \item \label{cor:nc_Cartan6b}%
    there are a discrete group~\(G\), a saturated outer Fell bundle
    \(\B=(B_t)_{t\in G}\) over~\(G\), and a \Star{}isomorphism
    \(\Cst_\red(\B)\congto B\) that maps~\(B_1\) isomorphically
    onto~\(A\).
  \end{enumerate}
  If the above conditions hold, then~\(B\) is simple and the
  group~\(G\) and the Fell bundle in~\ref{cor:nc_Cartan6b} are
  unique up to isomorphism.
\end{corollary}

\begin{proof}
  Each of the conditions in the assertion,
  except~\ref{cor:nc_Cartan0.5}, is equivalent to one of the
  conditions in Theorem~\ref{the:nc_Cartan}.
  Theorem~\ref{thm:characterisation_of_aperiodic_crossed_products}
  implies that \ref{cor:nc_Cartan0} and~\ref{cor:nc_Cartan0.5} are
  equivalent because simple algebras are prime.  The last part of
  the assertion follows from
  Theorem~\ref{thm:characterisation_of_aperiodic_crossed_products}
  as well as Corollary \ref{cor:dual_groupoid_well-defined} or
  Theorem \ref{the:Cartan_Hausdorff}.
\end{proof}

\begin{remark}
  The above corollary implies that, for any purely outer action of a
  discrete group \(G\) on a simple \(\Cst\)\nb-algebra~\(A\), the
  inclusion \(A\subseteq A\rtimes_\red G\) remembers the group and
  the crossed product \(A\rtimes_\red G\) is simple (simplicity is a
  classical result of Kishimoto~\cite{Kishimoto:Outer_crossed}).
  Both claims fail without outerness assumption (see
  Example~\ref{exm:different_dual_groupoids}).  Outerness is only
  sufficent for simplicity.  For instance, the inclusion
  \(\C\subseteq \C\rtimes_\red\mathbb{F}_n=
  \Cst_\red(\mathbb{F}_n)\) for the free group~\(\mathbb{F}_n\) on
  \(n \geq 2\) generators is not Cartan although
  \(\Cst_\red(\mathbb{F}_n)\) is simple.
\end{remark}

When \(B\defeq A\rtimes_\red G\) is an ordinary reduced crossed
product by a group action \(\alpha\colon G\to \Aut(A)\) and~\(A\) is
unital, then Zarikian showed in
\cite{Zarikian:Unique_expectations}*{Theorem~3.2} that
\ref{cor:nc_Cartan1b} and~\ref{cor:nc_Cartan1} in
Corollary~\ref{cor:simple_nc_Cartan} are equivalent to each other
and to~\(\alpha\) \emph{acting freely}, that is, if \(t\in
G\setminus \{1\}\) and \(d\in A\) are such that \(da=\alpha_t(a)d\)
for all \(a\in A\), then \(d=0\).  The case of a crossed product for
an outer group action is special because each slice is contained in
a global slice defined on all of~\(A\).  Hence we do not expect such
equivalences for inverse semigroup actions on non-simple
\(\Cst\)\nb-algebras.

\subsection{Commutative Cartan subalgebras}
\label{sec:Cartan_commutative}

Now let \(A\subseteq B\) be a \(\Cst\)\nb-inclusion where~\(A\) is
commutative.  Renault defined (commutative) Cartan subalgebras as
regular \(\Cst\)\nb-inclusions \(A\subseteq B\) where~\(A\) is a
maximal Abelian subalgebra of~\(B\) and there is a faithful
conditional expectation onto~\(A\) (see
\cite{Renault:Cartan.Subalgebras}*{Definition~5.1}).  He shows that
Cartan subalgebras \(A\subseteq B\) with separable~\(B\) are
equivalent to twists of topologically principal, Hausdorff, étale,
locally compact, second countable groupoids (see
\cite{Renault:Cartan.Subalgebras}*{Theorems 5.2 and~5.9}).  We now
use Theorem~\ref{the:nc_Cartan} to extend Renault's characterisation
to the non-separable case.  Then topologically principal, second
countable groupoids are replaced by effective groupoids.

\begin{corollary}
  \label{cor:Cartan}
  Let \(A\subseteq B\) be a regular \(\Cst\)\nb-subalgebra with a
  faithful conditional expectation \(E\colon B\to A\).
  Assume~\(A\) to be commutative.  The conditions in Theorems
  \textup{\ref{the:nc_Cartan}}
  and~\textup{\ref{the:Cartan_Hausdorff}}, which characterise Cartan
  subalgebras in the sense of Exel, are equivalent to the following
  conditions:
  \begin{enumerate}
  \item \label{cor:Cartan1}%
    the inclusion \(A\subseteq B\) is aperiodic;
  \item \label{cor:Cartan2}%
    \(A\) is a maximal Abelian subalgebra in~\(B\), that is,
    \(A'\subseteq A\);
  \item \label{cor:Cartan3}%
    \(A\) is a MASA, that is, a maximal Abelian \Star{}subalgebra
    in~\(B\);
  \item \label{cor:Cartan6}%
    there is a twist~\(\Sigma\) of an effective, Hausdorff, étale,
    locally compact groupoid~\(H\) such that
    \(B\cong \Cst_\red(H,\Sigma)\) with an isomorphism mapping~\(A\)
    onto~\(\Cont_0(H^0)\).
  \end{enumerate}
  The twisted groupoid \((H,\Sigma)\) in~\ref{cor:Cartan6} is unique
  up to isomorphism.
\end{corollary}

\begin{proof}
  Condition~\ref{cor:Cartan1} holds if and only if \(A\subseteq B\)
  is a Cartan inclusion by
  Theorem~\ref{thm:characterisation_of_aperiodic_crossed_products}.
  By Theorem~\ref{the:Cartan_Hausdorff}, a Cartan inclusion gives
  rise to an (essentially) unique Fell bundle
  \(\A=(A_\gamma)_{\gamma\in H}\) over a Hausdorff, étale, locally
  compact groupoid~\(H\) with unit space~\(\widecheck{A}\) such that
  \(B\cong \Cst_\red(H,\A)\) with an isomorphism restricting to
  \(A \cong\Cont_0(\widecheck{A}, \A)\).  Since
  \(A=\Cont_0(\widecheck{A})\), this Fell bundle must be a line
  bundle, and Fell line bundles over groupoids are equivalent to
  twisted groupoids.  Since the inclusion \(A\subseteq B\) is
  aperiodic, this groupoid is effective by Theorem
  \ref{thm:characterisation_of_aperiodic_crossed_products_via_topological_freeness}.
  Hence~\ref{cor:Cartan1} implies~\ref{cor:Cartan6}.  The converse
  follows from
  Theorem~\ref{thm:characterisation_of_aperiodic_crossed_products_via_topological_freeness}.
  Thus~\ref{cor:Cartan6} is equivalent to \(A\subseteq B\) being
  Cartan in the sense of Exel.

  Next we prove that~\ref{cor:Cartan2} in the present theorem is
  equivalent to condition~\ref{the:nc_Cartan3} in
  Theorem~\ref{the:nc_Cartan}, that is, to
  \(I' \cap \Mult(I B I) =Z\Mult(I)\) for each \(I\in\Ideals(A)\).
  This condition for \(I=A\) implies that~\(A\) is maximal abelian.
  Conversely, assume that there are an ideal \(I\in\Ideals(A)\) and
  \(\tau\in I' \cap \Mult(I B I)\) with
  \(\tau\notin Z\Mult(I)=\Mult(I)\).  Then there is \(f\in I\) with
  \(\tau f = f \tau\notin I\).  And \(\tau f \in B\).  Since~\(A\)
  is commutative, \(\tau f \in A'\): if \(a\in A\), then
  \(a\cdot (f \tau) = (a f)\tau = \tau (a f) = (\tau f) a\), that
  is, \(f\tau = \tau f\) commutes with~\(A\).  Hence~\(A\) is not
  maximal Abelian in~\(B\).

  It remains to show that \ref{cor:Cartan2} and~\ref{cor:Cartan3}
  are equivalent.  \ref{cor:Cartan2} implies~\ref{cor:Cartan3}
  because~\(A\) is a \(\Cst\)\nb-algebra.  For the converse,
  assume~\(A\) is MASA and let \(x\in A'\).  We need to show that
  \(x\in A\).  Since~\(A\) is a \Star{}subalgebra, \(x^* \in A'\) as
  well.  If~\(x\) is normal, then the \(\Cst\)\nb-subalgebra
  generated by \(x,x^*\) and~\(A\) is commutative.  Then maximality
  implies \(x\in A\).  For general \(x\in A'\), it follows that
  \(x^* x \in A\) because it is a normal element of~\(A'\).
  Represent~\(B\) faithfully on a Hilbert space~\(\Hils\) and let
  \(x=V\abs{x}\) be the polar decomposition of~\(x\).  So~\(V\) is a
  partial isometry with \(\ker V=(\abs{x}\Hils)^\bot=\ker x\).  Then
  \(\abs{x}=\sqrt{x^*x}\) is in~\(A\).  So~\(\abs{x}\) commutes
  with~\(x\) and with~\(x^*\).  Then
  \(V\abs{x} \abs{x} = x\abs{x} = \abs{x}V\abs{x}\).  And both
  \(V\abs{x}\) and~\(\abs{x}V\) vanish on the complement of the
  image of~\(\abs{x}\).  So~\(V\) commutes with~\(\abs{x}\).  And
  then~\(x\) is normal, as \(x=\abs{x}V\) implies that the range
  of~\(V\) contains \((\ker x)^\bot=(\ker \abs{x})^\bot\) and hence
  \(xx^*=\abs{x}VV^*\abs{x}=\abs{x}^2=x^*x\).  Thus \(x\in A\).
\end{proof}

\begin{remark}
  \ref{cor:Cartan2} and~\ref{cor:Cartan3} in Corollary
  \ref{cor:Cartan} are equivalent for an arbitrary
  \(\Cst\)\nb-subalgebra \(A\subseteq B\), and this is probably well
  known.  We have included a short proof for completeness, as we
  were not able to find a reference.
\end{remark}

\end{document}